\documentclass[a4paper,12pt,reqno]{amsart}
\usepackage{amsmath,amsfonts,amsthm,amssymb,amscd,enumerate,eucal,url,stmaryrd}
\usepackage[dvipsnames]{xcolor}
\usepackage{eurosym}
\usepackage{needspace}

\usepackage{graphicx}
\usepackage{bm}
\usepackage{pifont}
\usepackage{adjustbox}
\usepackage{tikz}
\usepackage{color}

\newtheorem{theorem}{Theorem}[section]

\newtheorem{proposition}[theorem]{Proposition}
\newtheorem{prop}[theorem]{Proposition}
\newtheorem{lemma}[theorem]{Lemma}
\newtheorem*{theorem*}{Theorem}

\theoremstyle{definition}
\newtheorem{definition}[theorem]{Definition}

\newtheorem{remark}[theorem]{Remark}
\newtheorem*{remarks}{Remarks}

\newcommand{\CP}{\mathbb{CP}}
\newcommand{\la}{\langle}
\newcommand{\ra}{\rangle}

\newcommand{\too}{\longrightarrow}

\newcommand{\RE}{{\mathbb R}}
\newcommand{\CX}{{\mathbb C}}
\newcommand{\R}{{\mathbb R}}

\newcommand{\Z}{\mathbb Z}

\newcommand{\bZ}{\mathbb{Z}}

\newcommand{\beq}{\begin{equation}}
\newcommand{\eeq}{\end{equation}}

\newcommand{\cA}{\mathcal{A}}
\newcommand{\cB}{\mathcal{B}}

\DeclareMathOperator{\Sp}{Sp}
\DeclareMathOperator{\GL}{GL}
\DeclareMathOperator{\SU}{SU}
\DeclareMathOperator{\SO}{SO}
\newcommand{\U}{{\mathrm U}}

\newcommand{\G}{{\mathrm G}}

\DeclareMathOperator{\Spin}{Spin}

\newcommand{\we}{\wedge}

\DeclareMathOperator\vol{vol}
\DeclareMathOperator\id{Id}
\DeclareMathOperator\Id{Id}

\newcommand{\frs}{\mathfrak{s}}

\newcommand{\su}{\mathfrak{su}}

\DeclareMathOperator\re{Re}
\DeclareMathOperator\im{Im}

\newcommand{\dc}{d^{\,c}}

\definecolor{darkGreen}{RGB}{0,104,0}

\numberwithin{equation}{section}

\title[Nearly parallel $\G_2$-manifolds]
{Nearly parallel $\G_2$-manifolds: formality and associative submanifolds}

\date{\today}

\author{Marisa Fern\'andez}
\address{Departamento de Matem\'aticas, Facultad de Ciencia y Tecnolog\'{\i}a, Universidad del Pa\'{\i}s Vasco  (UPV/EHU), 
Apartado 644, 48080 Bilbao, Spain}
\email{marisa.fernandez@ehu.es}
\author{Anna Fino}
\address{Dipartimento di Matematica \lq\lq Giuseppe Peano\\Universit\`a di Torino\\
Via Carlo Alberto 10\\
10123 Torino, Italy\\
\& Department of Mathematics and Statistics\\
Florida International University,  Miami  33199, USA}
 \email{annamaria.fino@unito.it, afino@fiu.edu}
\author{Alexei Kovalev}
\address{DPMMS, University of Cambridge\\
Centre for Mathematical Sciences,\\
Wilberforce Road\\
Cambridge CB3 0WB, UK}
\email{A.G.Kovalev@dpmms.cam.ac.uk}
\author{Vicente Mu\~noz}
\address{Facultad de Ciencias Matem\'aticas\\
Universidad Complutense de Madrid\\
Plaza Ciencias 3, Ciudad Universitaria\\
28040 Madrid, Spain}
\email{vicente.munoz@ucm.es}

\subjclass[2010]{53C25, 53C30, 53B25, 53C40}
\keywords{nearly parallel $\G_2$-structures, formality, associative $3$-folds, mapping tori}

\begin{document}
\begin{abstract} 
We construct new examples of  non-formal 
simply connected compact Sasaki--Einstein  $7$-manifolds. 
We determine the minimal model of the
total space of any fibre bundle over $\CP^2$ with fibre $S^1\times S^2$ or
$S^3/\Z_p$ ($p>0$), and we apply this to conclude that the Aloff--Wallach spaces are formal.
We also find examples of formal manifolds and non-formal manifolds, which are locally conformal parallel $\Spin(7)$-manifolds.

On the other hand, we construct associative
minimal submanifolds in the Aloff--Wallach spaces and in any regular Sasaki--Einstein  $7$-manifold; in particular,
in the space $Q(1,1,1)=\big(\SU(2) \times \SU(2) \times\SU(2)\big)/ \big(\U(1) \times\U(1)\big)$ with the natural $S^1$-family of 
nearly parallel $\G_2$-structures induced by the Sasaki--Einstein structure. In each of those cases, we obtain a family of non-trivial 
associative deformations.
\end{abstract}

\maketitle

\section{Introduction}\label{sec:intro}

A nearly parallel  $\G_2$-structure on a $7$-manifold
$P$  is given by a positive  $3$-form  $\varphi$ satisfying the differential equation
$d \varphi =  \tau_0 \star  \varphi$, for some non-zero constant  $\tau_0$, 
where  $\star$ is the Hodge star operator determined by the induced metric by $\varphi$.
Nearly parallel $\G_2$-manifolds were introduced as manifolds with weak holonomy $\G_2$ by Gray in \cite{Gray71}.

In  \cite{BFGK}, it was shown that there is a 1-1 correspondence between nearly
parallel $\G_2$-manifolds and Killing spinors. As a consequence, the metric $g$  induced by the $3$-form  $\varphi$ has to be Einstein 
with positive scalar curvature ${\mathrm{scal}}(g) = \frac{21}{8} \tau_{0}^2$. Thus,
if $(P, g)$ is complete, then $P$ is a compact manifold with finite fundamental group due to Myers' theorem.

Another equivalent description of nearly parallel  $\G_2$-structures is in terms of the 
metric cone $(\widehat P = P \times \R^+ , g^c =  r^2 g   + dr^2)$, which has to have holonomy contained in 
$\Spin(7)$,  viewed  as subgroup of $\SO(8)$. If $(P, g)$ is simply connected and not isometric to the standard sphere, then there are 
three possible cases: the holonomy of $(\widehat P, g^c)$ 
 is $\Sp(2)$ or, equivalently, $(P,g)$ is a $3$-Sasakian manifold; the holonomy of $(\widehat P, g^c)$ can be $\SU(4)$
 if and only if $(P,g)$ is a Sasaki--Einstein manifold; or the holonomy of $(\widehat P, g^c)$ coincides with $\Spin(7)$, in 
 which case the $\G_2$-structure is called {\em proper}. 
 We recall that these three cases correspond to the existence of a space of Killing spinors of dimension  $3, 2$ or $1$,  respectively (see \cite{FKMS}). 
 In particular, by  \cite{AS}, a  Sasaki--Einstein $7$-manifold has a canonical $S^1$-family of nearly parallel $\G_2$-structures
inducing the Sasaki--Einstein metric. 
 Moreover, by \cite{GS}, a $3$-Sasakian $7$-manifold carries a second Einstein metric
 induced by a proper nearly parallel $\G_2$-structure (see Section \ref{nearly-parallel G2} for details). 
 
 A celebrated result of Deligne, Griffiths, Morgan and Sullivan states that any compact K\"ahler manifold is formal \cite{DGMS}.
In the same spirit, formality of a manifold is related to the existence of suitable geometric structures on the manifold (see Section \ref{sect:formality} for details about formality).

Nearly parallel  $\G_2$-manifolds are in various ways rather similar  to nearly K\"ahler $6$-manifolds.
Such a manifold is also Einstein, it has a Killing spinor, and its metric cone 
has holonomy contained in $\G_2$. Simply connected compact manifolds of dimension $\leq 6$ are always formal
\cite{N-Miller}.
Even more, any simply connected compact  nearly K\"ahler manifold is formal by \cite{AT}. In contrast, an example of a non-formal simply connected compact nearly parallel $\G_2$-manifold is given  by  the simply connected compact  
Sasaki--Einstein  manifold $Q(1,1,1)=\big(\SU(2) \times \SU(2) \times\SU(2)\big)/ \big(\U(1) \times\U(1)\big)$  (see \cite[Theorem 3.2]{BFMT}). 
The space $Q(1,1,1)$ is the only example, known in the literature, of a nearly parallel $\G_2$-manifold satisfying all those properties.

The goal of this paper is two-fold. Our first  purpose is to study the formality of nearly parallel $\G_2$-manifolds. 
For simply connected compact $3$-Sasakian $7$-manifolds, that study was  already previously considered in \cite{FIM}, so that 
 we  focus on  $7$-dimensional Sasaki--Einstein  manifolds and proper nearly parallel $\G_2$-manifolds. 

In Section \ref{sect:formalG2} (Theorem~\ref{th:pezzo-nonformal-sasa-einst}), we construct new examples of non-formal simply connected compact Sasaki--Einstein $7$-manifolds.
Such a manifold  $S_k$ $(3 \leq k \leq 8)$ is obtained as the total space of a principal circle bundle  
 over the Sasaki--Einstein manifold $X_k = P_k \times S^2$, where  
 $P_k={\CP}{}^2\#\, k \overline{\CP}{}^2$ is a del Pezzo surface. We prove that the manifold $S_k$ is non-formal because it has a non-zero triple Massey product.
   
In Section \ref{sec:extra}, we study the formality of the  Berger space and of Aloff-Wallach spaces 
considering both as fibrations. More precisely, in Theorem \ref{th:formal-Berger}
we find  the minimal model of the total space of any principal $S^3$-bundle over $S^4$ and  show that such a space is formal. We use this result to prove  that the  Berger space 
and $S^7$ have the same minimal model, and so the Berger space is formal. 
For the Aloff-Wallach spaces $W_{k,l}$, by \cite{GuS}, we know that they are the total space 
of a $F$-fiber bundle over $\CP^2$, where $F=S^3/{\mathbb{Z}}_p$ with $p > 0$ if $p=k+l\not=0$, or $F=S^1 \times S^2$ if $k+l\not=0$. 
We determine the minimal model of the total space of such a fiber bundle, and we prove that such a space is formal. 
Moreover, we show that $W_{k,l}$ and the product manifold $S^2 \times S^5$ have the same minimal model (Theorem  \ref{th:formal-AlofW1} and Theorem \ref{th:formal-AlofW2}). 
The formality of the Berger space and Aloff-Wallach spaces  was previously proved by a different 
 method~\cite{K,KT1}. In fact, in ~\cite{K} it is proved that the Berger space is formal because it is {\em geometrically formal} (that is,
it has a Riemannian metric such that all wedge products of harmonic forms are harmonic). In ~\cite{KT1}, it is proved that
the spaces $W_{k,l}$ are formal since they are homogeneous spaces of Cartan type.
 
On the other hand, if $(P, \varphi)$ is a nearly parallel $G_2$-manifold, then the
torsion-free $\Spin(7)$-structure on the metric cone $(\widehat P, g^c)$  gives rise to  a canonical closed $4$-form which defines a 
calibration on $\widehat P$. A 3-dimensional submanifold $Y$ in $(P, \varphi)$ is called {\em associative} if $\varphi|_Y = \text{vol}_Y$ or,
equivalently, if its cone $\widehat{Y}$ in $\widehat P$ is calibrated by the
$4$-form (see  \cite[Lemma 2.10]{Kawai2}).
By \cite{BallM}, associative submanifolds in nearly parallel  $\G_2$-manifolds  $P$ are minimal and their infinitesimal deformations were considered by Kawai 
in \cite{Kawai2}. 

Associative  $3$-folds have been studied by 
Lotay \cite{Lotay} when $P$ is the  standard $7$-sphere, by Kawai \cite{Kawai} when $P$ is the squashed $7$-sphere,
and by Ball and Madnick  when $P$ is the Berger space \cite{BallM2},
the squashed $7$-sphere and the squashed exceptional
Aloff--Wallach space $W_{1,1}$ \cite{BallM3}. 
However, nothing seems to be known for other nearly-parallel $\G_2$-manifolds.

Noting the above, the  second goal of the paper is to provide constructions of associative submanifolds 
in a regular Sasaki--Einstein manifold and in the  Aloff--Wallach spaces. In Section \ref{3folds.SE}, we study such submanifolds 
in the case of a regular Sasaki--Einstein manifold (Propositions \ref{invar.assoc} and Theorem~\ref{Legendrian}).
In particular, we consider the case of the  compact Sasaki--Einstein space 
$Q(1,1,1)$, which is the total space of a  principal $S^1$-bundle over $S^2 \times S^2 \times S^2$.
We  first make explicit the 
canonical $S^1$-family  $\varphi_t$ of nearly parallel $\G_2$-structures  associated to the Sasaki--Einstein structure 
on $Q(1,1,1)$ (see Section \ref{ex:nearlyparallel}). Then, considering the Calabi--Yau cone 
over $Q(1,1,1)$,  we construct minimal associative $3$-folds in $Q(1,1,1)$, and we 
determine  a family of non-trivial associative deformations (Propositions~\ref{SL.torus} and \ref{assoc.torus}).  
We also study minimal associative $3$-folds in the  Aloff--Wallach spaces (Theorem~\ref{assoc.fibr}).

Given a nearly parallel $\G_2$-manifold $P$, the product manifold $P\times S^1$ and,  more generally,  the  mapping torus of $P$ by 
a  diffeomorphism preserving the nearly parallel $\G_2$-structure,  carries a natural locally conformal parallel $\Spin(7)$-structure. 
In Section \ref{sect:Spin(7)}, we use the mapping torus construction to the nearly-parallel  
 $\G_2$-manifolds $(Q(1,1,1), \varphi_t)$, the Berger space and the Aloff--Wallach spaces, and
we construct examples of formal and non-formal locally conformal parallel
$\Spin(7)$-manifolds (Propositions~\ref{prod.non-formal} -- Proposition \ref{prop:maptorus-AloffW}).
 
The  paper is organized as follows. In Section \ref{nearly-parallel G2} and Section \ref{sect:formality}, we briefly review properties on 
nearly parallel $\G_2$-structures and minimal models, respectively,  that we need for our results.
Then, in Section  \ref{sect:formalG2}  and Section  \ref{sec:extra},  we  examine  the formality of nearly parallel $\G_2$-manifolds. 
In  Section \ref{ex:nearlyparallel}, we  apply the construction of the canonical $S^1$-family of nearly parallel $\G_2$-structures    to the 
{Sasaki--Einstein}  manifold $Q(1,1,1)$. Section \ref{3folds.SE} is devoted to  the construction of minimal associative $3$-folds in 
$7$-dimensional regular Sasaki--Einstein manifolds and, in particular, in  $Q(1,1,1)$. Deformations of  minimal associative $3$-folds in any
$7$-dimensional regular Sasaki--Einstein manifold are also considered in Section \ref{3folds.SE}.
In  Section  \ref{assAW} we study the associative $3$-dimensional submanifolds  of  the  Aloff--Wallach spaces.  
Finally, Section \ref{sect:Spin(7)} contains  the  construction  of  
locally conformal parallel $\Spin(7)$-manifolds as mapping tori of each of the following spaces: $(Q(1,1,1), \varphi_t)$, the Berger space and the Aloff--Wallach spaces.

\section{Nearly parallel $\G_2$-structures}\label{nearly-parallel G2}

In this section, we recall the canonical $S^1$-family of nearly parallel $\G_2$-structures
that exists on any Sasaki--Einstein $7$-manifold \cite{AS, BGal, FK}, and the proper 
nearly parallel $\G_2$-structure  on any  $3$-Sasakian $7$-manifold \cite{FKMS, GS} (see \cite{BG2} for more details).

A $7$-manifold $P$ has a $\G_2$-structure if there is a reduction of the structure group of 
its frame bundle from ${\GL}(7,\mathbb{R})$ to the exceptional Lie group $\G_2$. 
By \cite{Gr69}, a $7$-manifold $P$ carries $\G_2$-structures
if and only if $P$ is orientable and spin.

The presence of a $\G_2$-structure is equivalent to the existence of a differential $3$-form $\varphi$ 
on $P$, which can be described locally as
$$
\varphi=e^{127}+e^{347}+e^{567}+e^{135}-e^{146}-e^{236}-e^{245},
$$ 
with respect to some basis $\{e^1,\dotsc, e^7\}$ of the (local) $1$-forms on $P$. Here, $e^{127}$ stands for $e^1\wedge e^2\wedge e^7$, and so on.

Since $\G_2 \subset \SO(7)$, a $\G_2$-structure $\varphi$ on $P$ determines a Riemannian metric $g=g_\varphi$ and an orientation on $P$ 
such that
\begin{equation}\label{metric}
 g (U,V)\,\vol=\frac 16 (U\lrcorner\varphi) \wedge (V\lrcorner\varphi) \wedge \varphi,
\end{equation}
for any vector fields 
$U, V$ on $P$, where $\vol$ is the volume form on $P$.

A $\G_2$-structure $\varphi$ on a $7$-manifold $P$ is 
said to be {\em nearly parallel} if there exists a non-zero real number $\tau_0$ such that
  \begin{equation}\label{g2np}
 d\varphi=\tau_0 \star\varphi,
  \end{equation}
where  $\star$ is the Hodge star operator determined by the metric $g_{\varphi}$ and the volume form on $P$ induced by $\varphi$. A $7$-manifold $P$ with a nearly parallel $\G_2$-structure is called a {\em nearly parallel $\G_2$-manifold}.

\subsection*{Sasaki--Einstein and $3$-Sasakian manifolds}\label{SE.and.3S}
Let us recall that  a Riemannian manifold $(S,g)$ of dimension $2n + 1$  is called {\em Sasakian} if its 
cone $(S\times{\mathbb{R}}^+, g^c = r^2 g+dr^2)$ is K\"ahler,
that
is the cone metric $g^c$ admits a compatible
integrable almost complex structure $J$ so that
$(S\times{\mathbb{R}}^+, g^c, J)$ is  K\"ahler.
 In this case the Reeb vector field $\xi\,=\,J \partial_r$ on $S$ is
a  unit Killing vector field. The corresponding $1$-form
$\eta$ defined by 
$\eta(U)\,=\,g(\xi,\,U)$, 	for any vector field $U$ on $S$, is a {\em contact form}, meaning
$\eta\wedge ({d} \eta)^n\,\not=\,0$ at every point of $S$.
The K\"ahler form on the cone can be expressed as
\begin{equation}\label{K.cone}
\omega^c=\frac12 d(r^2\eta),
\end{equation}
where we extended $\eta$ (and $\xi$) to
$S\times\RE^+$ by $\eta(U)=\frac1{r^2}g^c(J(r\partial_r),U)$.
Let $\nabla$ be the Levi--Civita connection of $g$. The endomorphism $\Phi$
of the tangent bundle $TS$ of $S$, given by $\Phi U=\nabla_U\xi$,
satisfies the following identities 
\begin{equation} \label{def:almostcontact-st} 
\Phi^2=-\Id+\eta\otimes\xi, \\\ g(\Phi U,\Phi V)=g(U,V)-\eta(U)\eta(V), 
\\\ d\eta(U,V)=2g(\Phi U,V),
\end{equation}
for any vector fields $U,V$ on $S$. If the integral curves of the Reeb vector field $\xi$ are closed,
hence circles, then $\xi$ integrates to a locally free isometric action of
$S^1$ on $(S,g)$. A Sasakian manifold is called {\em regular} when this latter $S^1$-action is free.

A Sasakian manifold $(S,g)$ of dimension $2n + 1$ is said to be {\em Sasaki--Einstein} if the cone metric $g^c \,=\, r^2 g+dr^2$
on $S\times{\mathbb{R}}^+$ is K\"ahler and Ricci-flat or, equivalently,
the restricted holonomy group of  $g^c$ is contained in
\mbox{$\SU(n+1)$.} Thus, the Sasakian metric $g$ is Einstein with Einstein constant $2n$.

 Compact regular Sasaki--Einstein manifolds 
are principal circle bundles over compact K\"ahler--Einstein manifolds with positive scalar curvature, and whose
 K\"ahler form defines a cohomology class proportional to an integral cohomology
class \cite{BG-2000,FK,Kobayashi} (therefore, the base manifold is a simply connected projective algebraic variety, more precisely a
Fano variety). The converse also holds. In fact, the following result is proved in \cite[Theorem 2.8]{BG-2000}.
Let $(X,g_X)$ be a compact K\"ahler--Einstein manifold with positive scalar curvature, and whose K\"ahler form $\omega_X$ 
defines an integral cohomology class, and let $S$ be the total space of the circle bundle 
\begin{equation}\label{epi}
S^1 \,\hookrightarrow\, S \,\stackrel{\pi}{\longrightarrow}\, X
\end{equation}
with Euler class $[\omega_X]\in H^2(X,\mathbb{Z})$. Then,
$S$ with the metric $g = \pi^*g_X + \eta\otimes\eta$ is a (regular)
Sasaki--Einstein manifold, whose contact form $\eta$ satisfies the equation $d\eta= 2 \pi^*\omega_X$, where 
$\pi$ is the projection in \eqref{epi}. 

Let $(S,g)$ be a Sasaki--Einstein manifold of dimension $7$ with contact
form~$\eta$. Then, 
according to \cite[pp.~723--724]{AS}, $S$ has an $S^1$-family of nearly
parallel $\G_2$-structures $\varphi_t$, which are given by
\begin{equation} \label{def:G2form-sasaki-eintein}
\varphi_{t} = \Omega \wedge  \eta + \cos t \,\Psi_{+} + \sin t  \, \Psi_{-}.
\end{equation} 
Here $\Omega$ is the horizontal
K\"ahler form related to the Ricci-flat K\"ahler form $\omega^c$ on the cone
$S\times{\mathbb{R}}^+$ via
$\Omega\we\eta=\frac12\partial_r\lrcorner(\omega^c \we \omega^c)|_{r=1}$,
equivalently $\Omega \,=\, \pi^*\omega_X$ in the case when $(S,g)$ is regular and $\pi \colon S  \to X$ is the projection
of the principal circle bundle.
Further, $\Psi = \Psi_{+} + i \Psi_{-}
=\partial_r\lrcorner\widehat\Psi|_{r=1}$ is a horizontal complex volume form,
where $\widehat\Psi$ denotes a holomorphic \mbox{$4$-form} of unit length
on $S\times{\mathbb{R}}^+$.
Now, a straightforward computation using \eqref{metric} shows that every
$\varphi_{t}$ induces the Sasaki--Einstein metric $g$ on~$S$.

\begin{remark}
Note that the expression for $\varphi_{t}$ given in \cite{AS} (where the authors write $\sigma_t$ instead of  $\varphi_{t}$) is
$$
 \varphi_{t} = - \Omega \wedge  \eta + \cos t \,\Psi_{+} + \sin t \, \Psi_{-},
$$
with $\eta$ the contact form of the Sasakian structure on $S$. The change that we made of the first term on the right-hand side of  \eqref{def:G2form-sasaki-eintein} is due to the following. 
By \eqref{def:almostcontact-st}, we have 
$d\eta(U,V)=2g(\Phi U,V)$, while in \cite{AS} the authors consider $d\eta(U,V) = 2 g(U, \Phi V)$, for all vector fields $U, V$ on $S$.
\end{remark}

Let us recall that a {\em $3$-Sasakian structure} is a collection of three Sasakian structures $(\phi_i, \xi_i, \eta_i, g)$ $(i\in\{1, 2, 3\})$ on a
$(4n+3)$-dimensional Riemannian manifold satisfying
quaternionic-like identities. More
precisely, a Riemannian manifold $(S, g)$ of dimension $4n+3$ is
called $3$-Sasakian if its cone $(S\times{\mathbb{R}}^+, g^c = r^2 g\,+\,dr^2)$ is hyperk\"ahler, or 
equivalently, the holonomy group of the cone metric $g^c$ is a
subgroup of $\mathrm{Sp}(n+1)$. In this case, the Reeb vector
fields $\xi_i=J_i\partial_t$ $(i=1,2,3)$ 
satisfy the relations: $g(\xi_i, \xi_j) = \delta_{ij}$ and
  $[\xi_i,\xi_j]=2\xi_k$ for $(i,j,k)$ a cyclic permutation of $(1,2,3)$.

If $(S, g)$ is a $7$-dimensional $3$-Sasakian manifold, then $S$ carries a second nearly parallel $\G_2$-structure whose 
underlying Einstein metric is such that its cone metric has holonomy equal to
$\Spin(7)$ (see \cite{FKMS, GS, NS}).
This second Einstein metric $\widetilde{g}$ on $S$ is given by 
\begin{equation} \label{second Einstein}
\widetilde{g}\, =\, \frac{1}{\sqrt{5}}\, g|_{\mathcal H} + g|_{\mathcal V}, 
\end{equation}
where ${\mathcal V}$ is the $3$-dimensional distribution ${\mathcal V} =\rm{span} \{\xi_1, \xi_2, \xi_3\}$, and ${\mathcal H} = {\mathcal V}^{\bot}$.
The explicit expression of the second nearly parallel $\G_2$-structure inducing the metric $\widetilde{g}$  is given in 
\cite[Prop.~2.4]{GS} (see also \cite{NS}).

\section{Minimal models and formality} \label{sect:formality}

In order to analyze the property of formality of nearly parallel $\G_2$-manifolds (Section \ref{sect:formalG2} and Section \ref{sec:extra})
and of $8$-manifolds with a locally conformal parallel $\Spin(7)$-structure (see Section \ref{sect:Spin(7)}), we review
here concepts about minimal models and formality (see \cite{DGMS, FHT, FM} for more details).

A {\it differential graded algebra} (or DGA) over the real numbers $\RE$, is a pair $(\cA,d)$ consisting
of a graded commutative algebra $\cA=\oplus_{k\geq 0} \cA^k$ over $\RE$ (that is, it is positively graded), and
a differential $d$ of degree $1$, satisfying the Leibnitz rule 
$d(a\cdot b) = (da)\cdot b +(-1)^{|a|} a\cdot (db)$, where
$|a|$ is the degree of $a$. 

The cohomology  $H^*({\cA})$ of a 
differential graded algebra $(\cA,\,d)$ is naturally a DGA with the 
product inherited from that on ${\cA}$ and with the differential
being identically zero. The DGA $({\cA},\,d)$ is {\it connected} if 
$H^0({\cA})\,=\,\RE$, and ${\cA}$ is {\em $1$-connected\/} if, in 
addition, $H^1({\cA})\,=\,0$. Henceforth we shall assume that all our DGAs are connected.
In our context, the main example of DGA is the de Rham complex $(\Omega^*(M),\,d)$
of a connected differentiable manifold $M$, where $d$ is the exterior differential.

Morphisms between DGAs are required to preserve the degree and to commute 
with the differential. A morphism $f:(\cA,d)\to (\cB,d)$ is a {\it quasi-isomorphism} if
the map induced in cohomology $f^*:H^*(\cA,d)\to H^*(\cB,d)$ is an isomorphism.

A DGA $(\mathcal{M},\,d)$ is said to be {\it minimal\/} if
\begin{enumerate}
 \item $\mathcal{M}$ is free as an algebra, that is, $\mathcal{M}$ is the free
 algebra $\bigwedge V$ over a graded vector space $V\,=\,\bigoplus_{i \geq 1} V^i$, and
 \item there is a collection of generators $\{x_\tau\}_{\tau\in I}$
indexed by some well ordered set $I$, such that
 $|x_\mu|\,\leq\, |x_\tau|$ if $\mu \,< \,\tau$, and each $dx_\tau$ is expressed in terms of preceding $x_\mu$, $\mu\,<\,\tau$.
 This in particular implies that $dx_\tau$ is always decomposable, i.e. $dx_\tau$ has no linear part.
 \end{enumerate}
We say that $(\bigwedge V,\,d)$ is a {\it minimal model} of the
differential graded commutative algebra $({\cA},\,d)$ if $(\bigwedge V,\,d)$ is minimal and there
exists a quasi-isomorphism $f\colon
{(\bigwedge V,\,d)}\longrightarrow {({\cA},\,d)}$.
In~\cite{Halperin}, it is proved that any connected DGA has a minimal model unique up to isomorphism. For $1$-connected
DGAs, a similar result was proved earlier in~\cite{DGMS}. 

A {\it minimal model\/} of a connected differentiable manifold $M$
is a minimal model $(\bigwedge V,\,d)$ for 
$(\Omega^*(M),\,d)$. If $M$ is  simply
connected, then the dual of the real homotopy vector
space $\pi_i(M)\otimes \RE$ is isomorphic to $V^i$, for any $i$ (see~\cite{DGMS}).

We say that a DGA $(\mathcal{A},\,d)$ is a {\it model} of a manifold $M$
if $(\mathcal{A},\,d)$ and $M$ have the same minimal model. Thus, if $(\bigwedge V,\,d)$ is the minimal
model of $M$, we have 
$$
 (\mathcal{A},\,d)\, \stackrel{g}\longleftarrow\, {(\bigwedge V,\, d)}\, \stackrel{f}\longrightarrow\, (\Omega^{*}(M),\,d),
$$
where $f$ and $g$ are quasi-isomorphisms.

A minimal algebra $(\bigwedge V,\,d)$ is called {\it formal} if there exists a
morphism of differential algebras $f \colon {(\bigwedge V,\,d)}\,\longrightarrow\,
(H^*(\bigwedge V),0)$ inducing the identity map on cohomology.
Also a differentiable manifold $M$ is called {\it formal} if its minimal model is
formal. 

Simply connected compact manifolds of dimension $\leq 6$ are always formal
\cite{N-Miller}, so dimension $7$ is the lowest dimension in which formality is an issue.

The formality of a minimal algebra $(\bigwedge V,\,d)$  is characterized  by the condition that $V$
can be decomposed into a direct sum $V\,=\, C\oplus N$ with $d(C) \,=\, 0$ with $d$ injective on $N$ and  such that every closed element in the ideal
$I(N)$ in $\bigwedge V$ generated by $N$ is exact \cite{DGMS}.
This characterization can be weakened using the concept of
$s$-formality introduced in \cite[Definition 2.1]{FM} as follows.

\begin{definition}\label{def:primera}
A minimal algebra $(\bigwedge V,\,d)$ is $s$-{\em formal} ($s> 0$) if for each $i\leq s$
the space $V^i$ of generators of degree $i$ decomposes as a direct
sum $V^i=C^i\oplus N^i$, where the spaces $C^i$ and $N^i$ satisfy
the three following conditions:
\begin{enumerate}
\item $d(C^i) = 0$,
\item the differential map $d\colon N^i\longrightarrow \bigwedge V$ is
injective, and
\item any closed element in the ideal
$I_s=I(\bigoplus\limits_{i\leq s} N^i)$, generated by the space
$\bigoplus\limits_{i\leq s} N^i$ in the free algebra $\bigwedge
(\bigoplus\limits_{i\leq s} V^i)$, is exact in $\bigwedge V$.
\end{enumerate}
\end{definition}

A differentiable manifold $M$ is $s$-{\em formal} if its minimal model
is $s$-formal. Clearly, if $M$ is formal then $M$ is $s$-formal, for any $s>0$.
The main result of \cite{FM} shows that  
sometimes the weaker
condition of $s$-formality implies formality.

\begin{theorem}\cite[Theorem 3.1]{FM} \label{fm2:criterio2}
Let $M$ be a connected and orientable compact differentiable
manifold of dimension $2n$ or $(2n-1)$. Then $M$ is formal if and
only if it is $(n-1)$-formal.
\end{theorem}

We will use also the following property, whose proof is exactly the same as the one given 
in \cite[Lemma 2.4]{FIM} for $7$-dimensional simply connected compact manifolds with $b_2\,\leq 1$.
\begin{lemma} \label{lem:$3$-formal}
Let $M$ be a $7$-dimensional compact manifold with $b_1(M)\,=0$ and $b_2(M)\,\leq 1$.
Then, M is $3$-formal and so formal.
\end{lemma}

In order to detect  non-formality, instead of computing the minimal
model, which is usually a lengthy process, one can use  
(triple) Massey products, which are obstructions to formality \cite{DGMS}
and they are defined in the following way.
Let $(\mathcal{A},\,d)$ be a DGA  and suppose that there are
 $[a_i]\,\in\, H^{p_i}(\mathcal{A})$, $p_i\,>\,0$,
$1\,\leq\, i\,\leq\, 3$, such that $a_1\cdot a_2$ and $a_2\cdot a_3$ are
exact. Write $a_1\cdot a_2=da_{1,2}$ and $a_2\cdot a_3=da_{2,3}$.
The {\it (triple) Massey product} of the classes $[a_i]$ is defined as
$$
\langle [a_1],[a_2],[a_3] \rangle \,=\, 
[ a_1 \cdot a_{2,3}+(-1)^{p_{1}+1} a_{1,2}
\cdot a_3] \in \frac{H^{p_{1}+p_{2}+ p_{3} -1}(\mathcal{A})}{[a_1]\cdot H^{p_{2}+ p_{3} -1}(\mathcal{A})+[a_3]\cdot H^{p_{1}+ p_{2} -1}(\mathcal{A})}.
$$

Note that a Massey product $\langle [a_1],[a_2],[a_3] \rangle$ on $(\mathcal{A},\,d_{\mathcal{A}})$
is zero (or trivial) if and only if there exist $\widetilde{x}, \widetilde{y}\in \mathcal{A}$ such that
$a_1\cdot a_2=d_{\mathcal{A}}\widetilde{x}$, \, $a_2\cdot a_3=d_{\mathcal{A}}\widetilde{y}$\, and
$$0\,=\,[ a_1 \cdot \widetilde{y}+(-1)^{p_{1}+1}\widetilde{x}\cdot a_3]
\,\in\, H^{p_{1}+p_{2}+ p_{3} -1}(\mathcal{A})\, .$$

\subsection*{Models of fibrations}\label{subsect:fibrations}
Let $F\to E \to B$ be a fibration of simply connected spaces. Let 
$({\mathcal{A}}_B, d_B)$ be a model  (not necessarily minimal) of
the base $B$, and let $(\bigwedge V_F, d_F)$ be a minimal model of the fiber $F$. By \cite[section 15]{FHT}, a model
of $E$ is the \emph{KS-extension} $({\mathcal{A}}_B\otimes \bigwedge V_F, D)$,
where $D$ is defined as
$Db=d_{B}b$, for $b\in {\mathcal{A}}_B$, and  $Dx= d_{F}x+ \Theta(x)$, for $x\in V_F$, and where 
$$
\Theta:V_F\to {\mathcal{A}}_B \oplus \big( \bigwedge\nolimits^+ V_F \otimes {\mathcal{A}}_B^+ \big).
$$
Note that $\Theta(x)$ is written in terms of the previous generators of $\bigwedge V_F$, by minimality.
The component of $\Theta$ in the first summand $\tau: V_F \to {\mathcal{A}}_B$ 
is the {\em transgression map}.
This is also true in the case that $F, E$ and $B$ are nilpotent
spaces and the fibration is nilpotent, that is $\pi_1(B)$ acts nilpotently in the homotopy groups $\pi_j(F)$ of the fiber. 

We will need two cases.
In the case that $E$ and $B$ are simply connected and $F=\SU(2)=S^3$ or 
$F=S^3/{\mathbb{Z}}_r$, with $r>0$, the fibration is nilpotent. 
The minimal model of $S^3$ is $(\bigwedge a,d)$, with $|a|=3$ and $da=0$. 
Both spaces $S^3$ and $S^3/{\mathbb{Z}}_r$ are rationally homotopy equivalent, because the quotient map 
$S^3 \to S^3/{\mathbb{Z}}_r$ is an isomorphism on cohomology and the fundamental group of 
$S^3/{\mathbb{Z}}_r$ is trivial after rationalization.
Therefore a fibration
$S^3/{\mathbb{Z}}_r\,\rightarrow \,E\,\to B$ is a rational $S^3$-fibration (that is, after rationalization of the spaces, it becomes
a fibration). The map $\Theta$ is such that $\Theta(a) \in{\mathcal{A}}_B^4$ is a closed 
element of degree $4$ defining the (rational) Euler class  $e(E)$ of the fibration.

The second case that we will need is for fibrations with $F=S^1\times S^2$.
Then the minimal model of $F$ is $(\bigwedge(b,c,x), d)$, where $|b|=1$, $|c|=2$, $|x|=3$, and 
the differential map is defined by $db=0$, $dc=0$ and $dx=c^2$.
If $B$ is simply connected then the fibration is nilpotent, and a model of $E$ is given as
$({\mathcal{A}}_B\otimes \bigwedge (b,c, x), D)$, where the differential $D$ is of the form
  $$
  Db=\tau(b), \,\,  Dc=b\, a_2'+\tau(c), \,\,  Dx=c^2+ c\, a_2+b\,a_3+\tau(x),
  $$
  where $\tau$ is the transgression map, and $a_2,a_2'\in {\mathcal{A}}_B^2$,
$a_3\in {\mathcal{A}}_B^3$. Here $\tau(b)$  is a closed element of degree $2$ defining the Chern class
of the $S^1$-bundle induced from $E$ via the fiberwise projection $S^1 \times S^2\to S^1$.

\section{On the formality of nearly parallel $\G_2$-manifolds} \label{sect:formalG2}

Let $P$ be a simply connected compact $7$-manifold with a nearly parallel $\G_2$-structure $\varphi$ inducing a metric $g$ such that 
$(P, g)$ is not isometric to the standard sphere $S^7$.
As already mentioned in the introduction, $(P, g)$ belongs to one of the following three classes:  
a {\em  proper} nearly parallel $\G_2$-manifold, a Sasaki--Einstein manifold  or a $3$-Sasakian manifold.

Formality of simply connected compact  $3$-Sasakian manifolds, of dimension $7$, was 
studied in \cite{FIM}, showing that   a simply connected compact $3$-Sasakian manifold $(S, g)$ of dimension $7$  is formal
if and only if its second Betti number $b_2(S)\leq 1$.

Regarding the formality of  Sasaki--Einstein manifolds, by  \cite[Theorem 3.2]{BFMT}, if  $\omega=\omega_1 +\omega_2 + \omega_3$ is  
the K\"ahler form on $S^2\times S^2\times S^2$, where $[\omega_1]$, $[\omega_2]$ and $[\omega_3]$ are the generators of the integral cohomology group of each of the 
$S^2$-factors on $S^2 \times S^2\times S^2$, then, 
the total space of the principal $S^1$-bundle
 $$
S^1 \,\hookrightarrow \,Q(1,1,1) \,\too\, S^2\times S^2\times S^2\, ,
 $$ 
with Euler class $[\omega]\in H^2(S^2\times S^2\times S^2,\mathbb{Z})$ is non-formal.

The above is a consequence of the following more general result.
\begin{proposition}\label{prop:nearly-non-formal} 
Consider the principal $S^1$-bundle
 $$
S^1 \,\hookrightarrow \, Q \,\too\, S^2\times S^2\times S^2\, ,
 $$ 
with Euler class $e_1a_1+e_2a_2+e_3a_3$, where $e_1,e_2,e_3\in \Z$, and $a_i$ is the generator
of $H^2(S^2,\Z)$ for the $i$-th copy of $S^2$, $i=1,2,3$. Then, $Q$ is formal if and only if $e_1e_2e_3= 0$.
\end{proposition}

\begin{proof}
We will determine a model of the $7$-manifold $Q$. A minimal model of 
$S^2 \times S^2 \times S^2$ is the differential algebra $(\bigwedge(a_1, a_2, a_3, x_1, x_2, x_3),\,d)$, where $|a_i|=2$ 
while $|x_i|=3$ with $1\leq i \leq 3$, and the differential $d$ is defined by $da_i\,=\,0$ and $dx_i\,=\,a_i^2$.
Therefore, a model of the total space of a fiber bundle
 $$
 S^1 \,\hookrightarrow \,Q \,\longrightarrow\, S^2\times S^2\times S^2
 $$ 
is the differential algebra over the vector space $V$ generated by the elements 
$y$ of degree $1$, $a_1, a_2, a_3$ of degree $2$, and  $x_1, x_2, x_3$ of degree $3$, and 
the differential $d$ is given by
 $$
 da_i\,=\,0\, , \quad dx_i\,=\,a_{i}^2\,, \quad 1\leq i\leq 3, \quad dy\,=\,e_{1} a_{1}+e_2 a_2+e_3 a_3\, ,
 $$
where $e_{1} a_{1}+e_2 a_2+e_3 a_3\,\in\, H^2(S^2\times S^2\times S^2,\,\mathbb Z)$ is the 
Euler class of the $S^1$-bundle.

If all $e_1=e_2=e_3=0$, then $(\bigwedge V, d)$ is a minimal DGA, and so it is the minimal model of $Q$. But $(\bigwedge V, d)$
is also the minimal model of $S^1\times S^2\times S^2\times S^2$, which is formal being the product of formal manifolds. 
Thus $Q$ is formal because its minimal model is so.

If not all $e_i$ are zero, we can assume $e_1 \neq 0$ (up to reordering). Then 
$e_1a_1= dy - e_2  a_2-  e_3 a_3$, so 
 \begin{align*}
 e_1^2dx_1 &=e_1^2a_1^2= (dy - e_2a_2- e_3a_3)^2 \\
 &=d(y\cdot dy)-2 \sum_{i=2}^3 e_i d(y\cdot a_i) + \sum_{i=2}^3 e_i^2 dx_i + 2 e_2e_3a_2\cdot a_3\,. 
\end{align*}
Thus, letting 
 $$
\widetilde{x}{}_1=x_1-  e_1^{-2}\left(y \cdot dy -2 \sum_{i=2}^3 e_i y\cdot a_i + \sum_{i=2}^3 e_i^2 x_i\right),
 $$
we have
 $$ 
d\widetilde{x}{}_1=  2 e_1^{-2}e_2 e_3 a_2\cdot a_3\,.
$$
 Then, the differential algebra $\big(\bigwedge(a_2, a_3,\widetilde{x}{}_1, x_2, x_3),\,d\big)$ is a model of $Q$. In fact, the map 
$f \colon (\bigwedge(a_2, a_3,\widetilde{x}{}_1, x_2, x_3),\,d) \to (\bigwedge(a_1, a_2, a_3, x_1, x_2, x_3, y),\,d)$ 
defined by
$f(a_i)=a_i$, $f(x_i)=x_i$ $(i= 2,3)$, and $f(\widetilde{x}{}_1)=x_1- e_1^{-2} \left(y \cdot dy -2 \sum_{i=2}^3 e_i y\cdot a_i + \sum_{i=2}^3 e_i^2 x_i \right)$ 
is a quasi-isomorphism. 

Let us assume that $e_2e_3=0$. In this case, $d\widetilde{x}{}_1=0$, and $(\bigwedge(a_2, a_3, x_2, x_3,\widetilde{x}{}_1),\,d)$ is a minimal differential graded algebra.
So $(\bigwedge(a_2, a_3, x_2, x_3,\widetilde{x}{}_1),\,d)$
is the minimal model of $Q$. Hence $Q$ is formal since  $(\bigwedge(a_2, a_3, x_2, x_3,\widetilde{x}{}_1),\,d)$ is also
the minimal model of $S^2\times S^2\times S^3$, which is formal being the product of formal manifolds.

Finally, if $e_2e_3\neq 0$, then 
$$
 a_2\cdot a_3= \frac{e_1^2}{2 e_2 e_3} d\widetilde{x}{}_1.
$$ 
We are going to show that $Q$ is non-formal because 
there exists a non-zero Massey product on $Q$. By \cite{BBFMT} we know that Massey products on a 
manifold can be computed by using any model for the manifold. Since $(\bigwedge(a_2, a_3, x_2, x_3,\widetilde{x}{}_1),\,d)$
is a model of $Q$, we have $H^*(Q) \cong H^*\big(\bigwedge(a_2, a_3, x_2, x_3,\widetilde{x}{}_1),\,d\big)$, so that
$H^1(Q)=0=H^6(Q)$, $H^2(Q)\cong \la [a_2],[a_3] \ra$, $H^3(Q)=0$, and by Poincar\'e duality for the $7$-manifold $Q$, $H^4(Q)=0$ 
and $H^5(Q)$ has dimension $2$. Moreover, the Massey product $\la [a_2],[a_2],[a_3]\ra$ is defined and
  $$
  \la [a_2],[a_2],[a_3]\ra =  \left[{\frac{e_1^2}{2e_2e_3}} a_2 \cdot\widetilde{x}{}_1 - x_2\cdot a_3\right].
  $$
This element in $H^5(Q)$ cannot be exact since there is no non-zero element 
$x\in \bigwedge^4 (a_2, a_3, x_2, x_3,\widetilde{x}{}_1)$
such that $dx = {\frac{e_1^2}{2e_2e_3}} a_2 \cdot\widetilde{x}{}_1 - x_2\cdot a_3$.
Moreover, the indeterminacy of the Massey product is zero because $H^3(Q)=0$. So 
$ \la [a_2],[a_2],[a_3]\ra \neq 0$, and hence $Q$ is non-formal.
\end{proof}

In order to exhibit further examples of non-formal simply connected Sasaki--Einstein manifolds, we consider a del Pezzo surface $P_k$, for $3\leq k\leq 8$,
that is the blow-up of the complex projective space ${\CP}{}^2$ at $k$ points,
 $$
P_k\,= \,{\CP}{}^2\#\, k \overline{\CP}{}^2=\,{\CP}{}^2\#\, \overline{\CP}{}^2\#\,  \overbrace{\cdots}^{k}
  \, \#\,\overline{\CP}{}^2, \quad 3\leq k\leq 8,
$$
where $\overline{\CP}{}^2$ is ${\CP}{}^2$ with the opposite of the standard orientation.
Then the de Rham cohomology of $P_k$ is
\begin{itemize}
 \item $H^0(P_k)=\langle 1\rangle$,
 \item $H^1(P_k)=0$,
 \item $H^2(P_k)=\langle a, a_1,\ldots, a_k \rangle$,
 \item $H^3(P_k)=0$,
 \item $H^4(P_k)=\langle \nu \rangle$,
\end{itemize}
where $\nu=a^2$ is the volume form, and $a$ is the integral cohomology class defined by the  K\"ahler form on $\CP{}^2$.
Among these cohomology classes, the following relations are satisfied
$$
a^2=-a_i^2=\nu, \,\,\, \text{for} \,\,\, 1\leq i \leq k, \quad a\cdot a_i= 0 =a_i\cdot a_j,
\,\,\, \text{for} \,\,\, 1\leq i, j \leq k \,\,\, \text{and} \,\,\, i\not=j.
$$
Now consider the 6-manifold
$
X_k =P_k\times S^2,$  $3\leq k\leq 8.
$

 \begin{theorem} \label{th:pezzo-nonformal-sasa-einst}
Let $S_k$ be the total space of the circle bundle  $S^1 \longrightarrow S_k \longrightarrow X_k =P_k\times S^2$, with Euler class 
$N(a-\sum_{i=1}^k \epsilon_i a_i + b)$, for 
$\epsilon_i>0$ such that $\sum \epsilon_i<1$, where $b$ is the generator of $H^2(S^2,\Z)$, and $N$ is a large
integer satisfying that $N\epsilon_i\in \Z$ for all $i$. 
Then, for $3\leq k \leq 8$ and some choice of Euler class as above, 
$S_k$ is a simply connected compact  Sasaki--Einstein manifold, with second Betti number $b_2 = k+1$, which is non-formal. 
\end{theorem}

\begin{proof}
It is standard that $a-\sum_{i=1}^k \epsilon_i a_i$ is a K\"ahler class for $P_k= {\CP}{}^2\#\, k \overline{\CP}{}^2$, when 
$\epsilon_i>0$ are small enough, since $a$ is the K\"ahler class of ${\CP}{}^2$ and $a_i$ represent the classes of the
exceptional divisors of the blow-up. The effective cone of
$P_k$ is generated by $a-a_i$ and $a_i$, $i=1,\ldots,k$. Hence the K\"ahler cone, which is its dual, is the open convex hull of $a$ and
$a-a_i$, $i=1,\ldots, k$. So $a-\sum_{i=1}^k \epsilon_i a_i$ is a K\"ahler class for $\sum \epsilon_i<1$,
and $a-\sum_{i=1}^k \epsilon_i a_i+b$ is a K\"ahler class for $X_k=P_k\times S^2$.
Now we choose $N$ so that $N\epsilon_i\in \Z$, hence 
$N(a-\sum_{i=1}^k \epsilon_i a_i +  b)$ is an integral cohomology class defined by a K\"ahler form on the product complex
manifold $X_k =P_k\times S^2$. 
Therefore, there is a circle bundle $S_k \longrightarrow X_k =P_k\times S^2$ with Euler class equal to 
$N(a-\sum_{i=1}^k \epsilon_i a_i + b)$, where  $b$ is the generator of $H^2(S^2,\Z)$.

Clearly $S_k$ is a $7$-dimensional simply connected, compact
manifold, with second Betti number $b_2=k+1$. For some choice of K\"ahler form, $S_k$ is
Sasaki--Einstein. Indeed, Tian and Yau in \cite{TY}, proved that there are K\"ahler--Einstein structures with $c_1 >0$
on any manifold $P_k\,= \,{\CP}{}^2\#\, k \overline{\CP}{}^2$, for $3\leq k \leq 8$.
Then, there exists a K\"ahler--Einstein metric on any manifold $X_k =P_k\times S^2$, for  $3\leq k \leq 8$ \cite{FKMS}.
Thus, according to Section \ref{nearly-parallel G2},  $S_k$ is a Sasaki--Einstein manifold.

By \cite{DGMS}, if $S_k$ has a non-zero Massey product, then $S_k$ is non-formal.
By \cite{BBFMT}, we know that Massey products on a manifold can be computed using any model for
the manifold. Since $X_k$ is a compact K\"ahler manifold, $X_k$ is formal. Thus, a model of $X_k$ is
$(H^*(X_k), 0)$, where $H^*(X_k)$ is the de Rham cohomology algebra of $X_k$, that is
\begin{align*} 
 H^0(X_k)&=\langle 1\rangle, \nonumber\\
 H^1(X_k)&=\,H^3(X_k)\,=\,H^5(X_k)\,=\,0\,, \nonumber\\
  H^2(X_k)&= \langle a, a_1,\ldots, a_k, b \rangle\,, \nonumber \\
 H^4(X_k)&= \langle a^2, a\cdot b, a_1\cdot b, \ldots, a_k\cdot b \rangle, \nonumber \\
 H^6(X_k)&= \langle a^2\cdot b \rangle. \nonumber
\end{align*}
Then, a model of $S_k$ is the differential graded algebra
$(\mathcal{A}, d)$, where $\mathcal{A}\,=\,H^*(X_k)\otimes \bigwedge(y)$, with $|y|\,=\,1$, $d(H^*(X_k))=0$
and $dy=N(a-\sum_{i=1}^k \epsilon_i a_i + b)$. Write $\widetilde{y}=\frac1N y$. Then, 
 $H^1(\mathcal{A}, d)\,=\,H^3(\mathcal{A}, d)=0$, $H^2(\mathcal{A}, d)\,=\,\langle [a], [a_1],\ldots, [a_k]\rangle$.

Using this model, we compute the Massey product $\la [a], [a], [a_1]\ra$. In this model $a\cdot a_i=0$ $(1 \leq i \leq k)$, and 
 $$
 (a-\sum_{i=1}^k \epsilon_i a_i + b)\cdot (a-\sum_{i=1}^k \epsilon_ia_i-b) =(1- \sum_{i=1}^k \epsilon_i^2)\,a^2,
$$
since $b^2 =0$. Then, $a\cdot a= (1- \sum_{i=1}^k \epsilon_i^2)^{-1} d\big( (a-\sum_{i=1}^k \epsilon_ia_i-b)\cdot \widetilde{y}\big)$.
So the Massey product $\la [a], [a], [a_1]\ra$ is defined, and
 \begin{align*}
\la [a], [a], [a_1]\ra&= \left[- (1- \sum_{i=1}^k \epsilon_i^2)^{-1} \Big( (a-\sum_{i=1}^k \epsilon_i a_i-b)\cdot \widetilde{y}\Big)\cdot a_1\right] \\
&= \left[-(1- \sum_{i=1}^k \epsilon_i^2)^{-1} (\epsilon_1\nu-b\cdot a_1) \cdot \widetilde{y}\right],
\end{align*}
which
is non-zero in $H^5(S_k)$. Therefore, $S_k$ is non-formal. Note that there is no indeterminacy of this Massey product, since
it lives in $[a] \cdot H^3(S_k) + [a_1] \cdot H^3(S_k)$, and we know that $H^3(S_k)\cong H^3(\mathcal{A}, d)=0$.
\end{proof}

Note that in \cite{FIM} there is an example of a $7$-dimensional regular simply connected Sasaki--Einstein manifold,
with second Betti number $b_{2}\geq 2$, which is formal, and so it does not admit any
$3$-Sasakian structure by \cite{FIM}. Such a manifold 
 is the total space of an $S^1$-bundle over the blow-up of the complex projective space
${\CP}{}^3$ at four points. Other examples of simply connected formal Sasaki--Einstein manifolds,
of dimension $7$, are the total space of a circle bundle over the 
K\"ahler--Einstein manifold $\CP^2 \times S^2$ \cite{BBFMT}, and the space $W_{1,1}$  
which, as Sasaki--Einstein manifold, is the 
total space of a circle bundle over the flag manifold $F(1,2)$ \cite{FKMS}.

\section{Formality of proper nearly parallel $\G_2$-manifolds} \label{sec:extra}

According to \cite{FKMS}, the only examples of proper nearly parallel $\G_2$-manifolds, whose underlying metric is homogeneous, 
are the squashed $7$-sphere $S^7_{sq}$, the Berger 
space $\mathcal B =\SO(5)/\SO(3)$ and the Aloff--Wallach spaces $W_{k,l}=\SU(3)/ S_{k,l}^1$\,.
The only compact non-homogeneous examples of proper nearly parallel $\G_2$-manifolds, known in the literature,
 are $7$-dimensional compact non-homogeneous $3$-Sasakian manifolds $(S,g)$ with the canonical variation metric 
 $\widetilde g$ of $g$ given by \eqref{second Einstein}.
  (Examples of such $3$-Sasakian manifolds are given in \cite{BGMR}.)

In this section we determine the minimal models of appropriate fibre bundles
over $\CP^2$ and show that $\mathcal B$ and $W_{k,l}$ are both formal. In
particular, we show that the Aloff--Wallach spaces and $S^2 \times S^5$ have
the same minimal model. It was previously known that as $\mathcal B$ is a
rational homology sphere it is geometrically formal for any choice of
Riemannian metric, i.e.\ have the property that wedge products of harmonic
forms are harmonic \cite{K}. The homogeneous spaces $W_{k,l}$ are of Cartan
type and therefore are formal   by ~\cite{KT1}. Thus the result that
$\mathcal B$ and $W_{k,l}$ are formal is not new but the proofs are new.

\subsection{The Berger space} \label{sec:extra-1}

Consider the usual action of $\SO(3)$  on $\RE^3 = \text{span} \{x, y, z \}$. This action extends to an action of $\SO(3)$ on the polynomial ring $\RE [x,y,z]$. 
Let $V_n \subset \RE [x,y,z]$ be the $\SO(3)$-submodule of homogeneous polynomials of degree $n$, and let ${\mathcal H}_n \subset V_n$ denote the 
$\SO(3)$-submodule of harmonic polynomials of degree $n$, an irreducible $\SO(3)$-module of dimension $2n+1$. Every finite dimensional irreducible 
$\SO(3)$-module is isomorphic to ${\mathcal H}_n$, for some $n$. The irreducible representation ${\mathcal H}_2$ of $\SO(3)$ has dimension 5, and so 
defines a non-standard embedding $\SO(3) \subset \SO(5)$.  The \emph{Berger space} is the compact homogeneous space $\mathcal B = \SO(5)/\SO(3),
$
given by the quotient of $\SO(5)$ by the copy of $\SO(3)$ embedded in $\SO(5)$ via the irreducible representation ${\mathcal H}_2$ of $\SO(3)$.
The space $\mathcal B$ has a metric such that the holonomy of its cone metric is $\Spin(7)$  \cite{B}.
The proper nearly parallel $\G_2$-structure on $\mathcal B$ is given explicitly in \cite[Subsection 2.4.1]{BallM2}.

Berger \cite{Berger} proved that 
$\mathcal B$ is a rational homology sphere with $H^4(\mathcal B,\Z) = \Z_{10},$
and it has the cohomology ring of an $S^3$-bundle over~$S^4$. Therefore, $\mathcal B$ and the sphere $S^7$ have the same minimal model
(and in particular, it is formal). 
In the following theorem, we 
determine the minimal model of the total space of a principal $S^3$-bundle over~$S^4$, and we show that such a space is formal.
We apply this to conclude again that the space $\mathcal B$ and $S^7$ have the same minimal model, and hence $\mathcal B$ is formal.

\begin{theorem}\label{th:formal-Berger}
Consider an $S^3$-bundle $S^3 \,\rightarrow \,P \,\rightarrow\, S^4$.
Then $P$ is formal.  In particular, the Berger space $\mathcal B = \SO(5)/\SO(3)$ is formal.
\end{theorem}

\begin{proof}
Let  $e [\omega]$ be the Euler class of the bundle, where $e\in\mathbb{Z}$, and 
$[\omega]$ is the generator of the integral cohomology group $H^4(S^4,\mathbb{Z})$. 
The minimal model of  $S^4$ is the differential graded algebra $(\bigwedge(a,u), d)$, with 
$|a|=4$, $|u|=7$, $da=0$ and $du=a^2$. The minimal model of $S^3$ is 
$(\bigwedge b, d)$, where $|b|=3$ and $db=0$. Then, according to Section \ref{sect:formality}, a model of $P$ is 
$(\bigwedge(a,u,b), D)$, with 
$$
Da=0, \quad Du=a^2, \quad Db=e\, a.
$$

If $e=0$, then the DGA $(\bigwedge(a,u,b), d)$ is minimal, and so it is the minimal model of $P$. But 
$(\bigwedge(a,u,b), D)$ is the minimal model of $S^3\times S^4$, which is formal being the product of two formal manifolds. Therefore,
$P$ is formal. 

Suppose now that $e\neq 0$. In this case we have $du=a^2=e^{-1}D(b\cdot a)$. So the element $\widetilde u=u-e^{-1}b\cdot a$ has degree $7$ and $D\widetilde u=0$. 
Then the DGA  $(\bigwedge(a,\widetilde u,b), D)$ is also a model of $P$, because it is quasi-isomorphic to $(\bigwedge(a,u,b), D)$. Moreover, 
$(\bigwedge \widetilde u, 0)$ is the minimal model of $P$. In fact, $(\bigwedge\widetilde u, 0)$ is a minimal DGA, and a model of 
$P$ because, taking into account that  $Db=e\, a$, the map $f \colon (\bigwedge\widetilde u, 0) \to (\bigwedge(a,\widetilde u,b), D)$ 
given by $f(\widetilde u) = \widetilde u$
is a quasi-isomorphism. Therefore, $P$ is formal since $(\bigwedge\widetilde u, 0)$ is the minimal model of the sphere $S^7$, which is formal.

 In \cite{GKS} it is proved that the Berger space is diffeomorphic to  
the $S^3$-bundle over $S^4$ with Euler class $- 10\,[\omega]$. Thus $\mathcal B$ and $S^7$ have the same minimal model, and 
consequently $\mathcal B$ is formal.
\end{proof}

\subsection{The Aloff--Wallach spaces} \label{AW}

Let $k, l \in \mathbb{Z}$ be non-zero, co-prime integers,
and $S_{k,l}^1$ be a circle subgroup of $\SU(3)$ consisting of elements of the form
$$\left( \begin{array}{ccc}
e^{ik\theta} & 0 & 0 \\
0 & e^{il\theta} & 0 \\
0 & 0 & e^{im\theta} \end{array} \right),$$
where $k+l+m=0$. The Aloff--Wallach space $
W_{k,l}=\SU(3) / S_{k,l}^1
$
is the quotient of $\SU(3)$ by this circle subgroup \cite{AW}. Note that there are examples of different pairs $(k, l)$ such that the corresponding 
Aloff--Wallach spaces are homeomorphic but not diffeomorphic \cite{KS}.  
The spaces $W_{k,l}$ are called generic if $\{k, l, -(k + l) \}$ is different from $\{1,1, -2\}$ or $\{ 1,-1, 0\}$.  We will denote by $W_{1,1}$ and  $W_{1, -1}$  the two  exceptional Aloff--Wallach spaces.

By \cite{PP}  (see also \cite{Wang}) all the  spaces $W_{k,l}$ admit two homogeneous Einstein metrics.  
If $(k,l)=(1,1)$ one of those metrics is the $3$-Sasakian structure on the space 
$W_{1,1}$ mentioned 
in Section \ref{sect:formalG2}, and the other is induced by a proper  homogeneous nearly parallel $\G_2$-structure.  
If $(k, l ) = (1, -1)$, the space $W_{1,-1}$  admits only one  proper  homogeneous nearly parallel $\G_2$-structure,  up to homotheties \cite{PP}.
On the generic Aloff-Wallach spaces  the  two metrics are induced by proper  homogeneous nearly parallel $\G_2$-structures \cite{BFGK}, which  
by \cite{PP, Reidelgeld} are only two,  up to homotheties. The expressions 
of those two $\G_2$-structures are given in \cite{BallOl, CMS}. 

The manifold $W_{k,l}$ is simply connected with $H^2(W_{k,l}, \Z)\cong \Z$ and $H^3(W_{k,l}, \Z)=0$ (see \cite{KS}). 
Thus, $b_1(W_{k,l})=b_3(W_{k,l})=0$ and $b_2(W_{k,l})=1$. Hence, $W_{k,l}$ is formal by Lemma \ref{lem:$3$-formal}. 

In \cite{GuS}
(see also \cite{BallOl}) it is shown that there is a canonical fibration
$$ 
\pi: W_{k,l} \rightarrow \CP^2,
$$
whose fibers are the lens spaces $S^3/\mathbb{Z}_{\vert k+l \vert}$ if $k+l \neq 0$, or $S^1 \times S^2$ if $k+l=0$. 
In the two following theorems, we determine the minimal model of the total space 
of a $F$-fiber bundle over $\CP^2$, where $F=S^1 \times S^2$, or $F=S^3/{\mathbb{Z}}_p$ with $p > 0$, 
and we prove that such a space is formal.
In particular, we show that $W_{k,l}$ and $S^5\times S^2$ have the same minimal model.

\begin{theorem}\label{th:formal-AlofW1}
Let $P$ be the total space of an $S^3/{\mathbb{Z}}_p$-bundle $S^3/{\mathbb{Z}}_p \,\rightarrow \,P \,\rightarrow\,  \CP^2$ with $p > 0$. 
Then $P$ is formal. 
In particular, if $k+l\not=0$, the Aloff--Wallach space $W_{k,l}$ is formal, 
and $W_{k,l}$ and the product manifold $S^2\times S^5$ have the same minimal model.
\end{theorem}

\begin{proof}
We will determine the minimal model of $P$ and show that it is formal. 
According to Section \ref{sect:formality},  
the fibre bundle $S^3/{\mathbb{Z}}_p \,\rightarrow \,P \,\rightarrow\,  \CP^2$ is a rational $S^3$-fibration. 
Let $e\,a^2$ be its (rational) Euler class, where $e\in\mathbb{Q}$ and $a=[\omega]$ is the generator of the integral cohomology group 
$H^2( \CP^2,\mathbb{Z})$.  By \cite{FHT}, if $({\mathcal{A}}, d_{{\mathcal{A}}})$
is a model of $\CP^2$, we have that $({\mathcal{A}}\otimes \bigwedge u, d)$, with $\vert u\vert =3$, 
$d|_{\mathcal{A}} = d_{{\mathcal{A}}}$ and $du = e\,a^2$, is a model of $P$.

The minimal model of $\CP^2$ is the differential graded algebra $(\bigwedge(a,x), d)$, where $|a|=2$, $|x|=5$, $da=0$ and
$dx=a^3$. Therefore, the KS-model of $P$ is $(\bigwedge(a,x, u), d)$, with $du=e\,a^2$. 

If $e=0$ then the differential graded algebra $(\bigwedge(a,x, u), d)$ is minimal, and so it is the minimal model of $P$. 
Moreover, $(\bigwedge(a,x, u), d)$ is the minimal model of $S^3\times \CP^2$, which is formal being the product of two formal manifolds.
Thus, $P$ is formal.

Suppose now that $e\neq 0$. In this case we have $a^2 = e^{-1} du$, and so the element of degree five  $\widetilde x=x-e^{-1} a\cdot u$ 
is such that $d\widetilde x=0$. Clearly $(\bigwedge(a,\widetilde x, u), d)$ is a minimal DGA, and a model of $P$ because the map 
$f \colon (\bigwedge(a,\widetilde x, u), d) \to (\bigwedge(a,x, u), d)$ given by $f(a)=a$, $f(\widetilde x) = x-e^{-1} a\cdot u$ and $f(u)=u$
is a quasi-isomorphism. Therefore, $(\bigwedge(a,\widetilde x, u), d)$ is the minimal model of $P$. Thus, $P$ is formal since
$(\bigwedge(a,\widetilde x, u), d) = (\bigwedge \widetilde x \otimes \bigwedge(a,u), d)$ is
 the minimal model of $S^5\times S^2$, which is formal.
 
All this shows not only that $P$ is formal but also the minimal model of $P$. Indeed, 
 the minimal model of $P$ is either the minimal model of $S^3\times \CP^2$ or the minimal model of $S^5 \times S^2$.
Thus, if the third Betti number of $P$ is $b_3(P)=1$, then $P$ and $S^3 \times \CP^2$
have the same minimal model, while if $b_3(P)=0$, then the minimal model of $P$ is
the minimal model of $S^5 \times S^2$. Therefore, for $k+l\not=0$, the space $W_{k,l}$
and $S^5\times S^2$ have the same minimal model since $b_3(W_{k,l})=0$.
\end{proof}

\begin{theorem}\label{th:formal-AlofW2}
The total space of an $S^1 \times S^2$-bundle $S^1 \times S^2 \,\rightarrow \,P \,\rightarrow\,  \CP^2$ 
is formal.  
In particular,  for $k\not=0$, the Aloff--Wallach space $W_{k,-k}$ is formal, and $W_{k,-k}$
 and the product manifold $S^2\times S^5$ have the same minimal model.
\end{theorem}

\begin{proof}
We know that the minimal model of $\CP^2$ is the differential graded algebra $(\bigwedge(a,x), d)$, where $|a|=2$, $|x|=5$, $da=0$ and
$dx=a^3$. Moreover, the minimal model of $S^1\times S^2$ is $(\bigwedge(b,c,y), d)$,
where $|b|=1$, $|c|=2$, $|y|=3$, and the differential map is defined by $db=0$, $dc=0$ and $dy=c^2$. 
Hence the KS-model of $P$ is of the form
 $$
\big(\bigwedge(a,x,b,c,y),D\big), \,\,\, Da=0, \, Dx=a^3, \, Db= e\,a, \, 
Dc = g\,ab, \, Dy=c^2+f \, a^2+h\,ac,
$$
for some $e,f,g,h\in \Z$. First note that taking $\tilde c=c+\frac12 h\, a$, we have $D\tilde c = g\,ab$ and $Dy=\tilde c^2+\tilde f a^2$,
for some $\tilde f$. Therefore we can assume in the above that $h=0$. We do that without changing the names of the variables.
By noting that $D^2(y)=2g\, abc=0$, we get $g=0$. Altogether we get the model of $P$ as
 $$
 \big(\bigwedge(a,x,b,c,y),D\big), \,\,\, Da=0, \, Dx=a^3, \, Db= e\,a, \,  Dc = 0, \, Dy=c^2+f \, a^2.
$$

If $e=f=0$ then $(\bigwedge(a,x,b,c,y),D)$ is the minimal model of $S^1\times S^2\times \CP^2$, which is formal being the product 
of three formal manifolds. Hence $P$ is formal.

Suppose now that $e\neq 0$ and $f=0$. In this case, $a= e^{-1} Db$. Then the element $\widetilde x$ 
of degree $5$ given by $\widetilde x=x-e^{-1} a\cdot u$ is such that $D\widetilde x=0$. Thus,
$(\bigwedge(a,\widetilde x,b,c,y),D)$ is a model of $P$ because this DGA is quasi-isomorphic
to $(\bigwedge(a, x,b,c,y),D)$. Moreover, $(\bigwedge(\widetilde x, c, y), D)$ is a minimal DGA, and a model of $P$. In fact,
taking into account that $a= e^{-1} Db$, one can check that the map 
$f \colon (\bigwedge(\widetilde x, c, y), D) \to (\bigwedge(a,\widetilde x,b,c,y),D)$ given by $f(\widetilde x) = \widetilde x$,  $f(c)=c$
and $f(y)=y$ is a quasi-isomorphism. Therefore, $(\bigwedge(\widetilde x, c, y), D)$ is the minimal model of $P$. Thus, $P$ is formal since
$(\bigwedge(\widetilde x, c, y), D)$ is the minimal model of $S^5\times S^2$, which is formal.

If $e\neq0$ and $f \neq 0$, as before we determine the minimal model of $P$. Take the elements $\widetilde x=x- e^{-1} b \cdot a^2$
and $\widetilde y=y- f e^{-1} a \cdot b$ of degree $5$ and $3$, respectively. Then, we get the model 
$(\bigwedge(a,\widetilde x,b,c,\widetilde y),D)$ of $P$ with $Da=0$, $d\widetilde x=0$, $Db= e\,a$,  $Dc = 0$ and $D\widetilde y=c^2$. 
Consider the differential graded algebra $(\bigwedge(\widetilde x, c, \widetilde y), D)$, which is a minimal DGA, and a model of $P$. In fact, the map 
$f \colon (\bigwedge(\widetilde x, c, \widetilde  y), D) \to (\bigwedge(a,\widetilde x, b,c,\widetilde y),D)$ given by 
$f(\widetilde x) = \widetilde x$,  $f(c)=c$
and $f(\widetilde y)=\widetilde y$ is a quasi-isomorphism. Therefore, $(\bigwedge(\widetilde x, c, \widetilde y), D)$ is the minimal model of $P$. 
Thus, $P$ is formal since
$(\bigwedge(\widetilde x, c, \widetilde y), D)$ is the minimal model of $S^5\times S^2$, which is formal.

Finally, suppose that $e=0$ and $f\neq 0$. Then the model $(\bigwedge(a,x,b,c,y),D)$ of $P$ is such that
$Da=0$, $Dx=a^3$, $Db= 0$, $Dc = 0$ and $Dy=c^2+f \, a^2$. This implies that the differential graded algebra
$(\bigwedge(a,x,b,c,y),D)$ is minimal, and so it is the minimal model of $P$. To show that $P$ is formal we proceed as follows. 
First note that $\bigwedge(a,x,b,c,y) = \bigwedge b \otimes \bigwedge(a,x,c,y)$. Clearly, $(\bigwedge b,0)$ is the minimal model 
of $S^1$, which is formal, and $(\bigwedge(a,x,c,y), D)$ is the minimal model of the total space $X$ of an $S^2$-bundle
$$
S^2 \,\rightarrow \,X \,\rightarrow\, \CP^2 .
$$
The de Rham cohomology of $X$ is $H^1(X)=H^3(X)=H^5(X)=0$, $H^2(X)=\la [a],[c]\ra$, 
$H^4(X)=\la [a]^2, [a]\cdot [c]\ra$ and $H^6(X)=\la [a]^2\cdot [c]\ra$. Thus, the $6$-manifold $X$ is formal since $b_1(X)=0$. In fact,
denote by $V$ the graded vector space generated by the elements $a,x,c$ and $y$. Because $V$ is a graded vector space,
 we can consider the space $V^i$ of generators of degree $i$, which decomposes as a direct sum $V^i = C^i \oplus N^i$, with 
$C^1 \,=\, N^1\,=\,0$, $C^2\,=\,\langle a, c \rangle$ and $N^2\,=\,0$. Thus, according to 
Definition \ref{def:primera}, the manifold
$X$ is  $2$-formal and, by Theorem \ref{fm2:criterio2}, $X$ is formal.

Therefore, the minimal model $(\bigwedge(a,x,b,c,y),D)$ of $P$ is the minimal model of the product manifold $S^1 \times X$,
which is formal being the product of two formal manifolds. Hence $P$ is formal.

Thus, if $P$ is the total space of an $S^1 \times S^2$-bundle over $\CP^2$, then $P$ and 
$S^5\times S^2$ have the same minimal model, 
or the minimal model of $P$ is the minimal model of the manifold $M$, where $M=S^1 \times S^2 \times \CP^2$, or $M=S^1 \times X$. 
So, for $k\not=0$, $W_{k,-k}$ and $S^5\times S^2$ have the same minimal model since $b_1(W_{k,-k})=0$.
\end{proof}

\begin{remark} \label{def-new-AW}
In \cite{EZ} it is proved that, for $k, l$ co-prime integers, the space $W_{k,l}= \SU(3) / S_{k,l}^1$
can be described as follows. Let $X$ be the total space of the $S^2$-bundle 
$S^2 {\longrightarrow} X\stackrel{\pi}{\longrightarrow} \mathbb{CP}^2$ with Pontryagin class $p_1= - 3$ and Stiefel-Whitney class $w_2 \not=0$.
Then, $W_{k,l}$ is the total space of the circle bundle 
\begin{equation}\label{def:AW-S1-bundle}
S^1{\longrightarrow} \, W_{k,l}\stackrel{\varpi}{\longrightarrow} X
\end{equation}
determined by the Euler class $k a+ l b$, where $a, b$ are the generators of $H^2(X, \Z)\cong \Z^2$, and the base space $X= \SU(3) / T^2$.

Note that the space $W_{k,l}$ thus defined can be also considered as the aforementioned $F$-fiber bundle
$F\, {\longrightarrow} W_{k,l} {\longrightarrow}\, \mathbb{CP}^2$, where $F= S^1\times S^2$, or $F=S^3/{\mathbb{Z}}_p$ with $p\not=0$.
Indeed, for $z\in \mathbb{CP}^2$, consider $X_{z} = \pi^{-1}(z)$ and $W_{k,l}(z)=\rm{pr}^{-1}(z)$, where $\rm {pr} = \pi \circ \varpi$. If we restrict to each 
$X_{z}$, this circle bundle $S^1\, {\longrightarrow}  W_{k,l}(z)\, {\longrightarrow} X_{z}$ must be $W_{k,l}(z) = S^1 \times S^2$ if the Euler class is $e=0$, or a lens space
$S^3/{\mathbb{Z}}_e$ if the Euler class is $e\neq 0$. Varying over all $z\in \mathbb{CP}^2$, we have a fiber bundle
$F{\longrightarrow} W_{k,l} {\longrightarrow} \mathbb{CP}^2$, where $F= S^1 \times S^2$ or $F=S^3/{\mathbb{Z}}_e$.
\end{remark}

\section{The nearly parallel $\G_2$-manifold $Q(1,1,1)$}\label{ex:nearlyparallel}

In this section, we consider 
the regular Sasaki--Einstein $7$-manifold $Q(1,1,1)$ and we work out explicitly, in coordinates,
the $S^1$-family of nearly parallel $\G_2$-structures on it, which we will use in
Section \ref{3folds.SE} and Section \ref{sect:Spin(7)}.

We start from the K\"ahler manifold $X=S^2 \times S^2\times S^2$ with
K\"ahler form
\begin{equation}\label{omega-X}
\omega\,=\,\omega_1 +\omega_2 + \omega_3\, ,
\end{equation}
where $\omega_1$, $\omega_2$ and $\omega_3$ are the generators of the integral cohomology group for each of the 
$S^2$-factors on $S^2 \times S^2\times S^2$.
Let $M$ be the total space of the principal $S^1$-bundle
\begin{equation}\label{Q.bundle}
S^1 \,\hookrightarrow \,M \,\overset{\pi_M}\too\, X=S^2\times S^2\times S^2\, ,
\end{equation}
with Euler class $[\omega]\in H^2(X,\mathbb{Z})$. Then, by application of \cite[cf.~Theorem 2.8]{BG-2000}, $M$ is a simply connected compact Sasaki--Einstein manifold,
with contact form $\eta$ such that $d \eta=2\pi_M^*(\omega)$. From
\cite[\S 4.2]{BFGK} (see also \cite{FK, FKMS}), we know that $M$ is the homogeneous space
\begin{equation}\label{q111}
M= Q(1,1,1) = \Big(\SU(2) \times \SU(2) \times\SU(2)\Big)/ \Big(\U(1) \times\U(1)\Big).
\end{equation}
Throughout this section, the notation $M$ and $X$ will always mean the
manifolds defined in \eqref{Q.bundle} and~\eqref{q111}.

We apply the Kobayashi construction \cite{Kobayashi} 
to determine the
contact form for the Sasaki--Einstein metric of the principal circle bundle $M$.
Let  $\theta_j\in (0,\pi)$ and $\phi_j\in (0,2\pi)$, $j=1,2,3$, be the
standard spherical coordinates on each of the $S^2$-factors in $X$ and
$z_j=\cot(\theta_j/2)e^{i\phi_j}\in\CX$
a complex coordinate defined via the stereographic projection.
The Fubini--Study metric of volume 1 on the $j$-th $S^2$ factor has the K\"ahler
form  $\omega_j = \frac{1}{4\pi} d \dc \log (1 + z_j\bar{z}_j) =
\alpha_{2j-1}\we\alpha_{2j}$, where we used an orthonormal co-frame field
\begin{equation} \label{1-forms}
\alpha_{2j-1}= \frac{-1}{2\sqrt{\pi}} \, d\theta_j, \qquad
\alpha_{2j}= \frac1{2\sqrt{\pi}} \, \sin\theta_j\, d\phi_j, \qquad
j=1,2,3,
\end{equation}
on the open dense region $U_X\subset X$ defined by $z_j\neq 0$ for all $j$.
The K\"ahler form~\eqref{omega-X} on $X$
restricts to an exact form on $U_X$
\begin{equation}\label{omega}
\omega =
\alpha_1 \wedge\alpha_2 + \alpha_3 \wedge\alpha_4 + \alpha_5 \wedge\alpha_6 =
\frac{1}{4\pi} \sum_{j=1}^{3} d(\cos\theta_j \, d\phi_j).
\end{equation}
The principal $S^1$-bundle $M$ trivializes over $U_X$,
$\pi_M^{-1}(U_X)\cong U_X\times S^1$. Let $s\in (0, 4\pi)$ denote a coordinate
on the fibre $S^1$ and consider on $\pi_M^{-1}(U_X)$ a $1$-form
\begin{equation}\label{eta}
\eta =
\frac{1}{2\pi} \bigl( - ds + \sum_{j=1}^3 \cos\theta_j\, d\phi_j \bigr).
\end{equation}
Noting that $[\omega]\in H^2(X,\bZ)$ is an integral cohomology class, it can
be checked by following the general construction in \cite[\S 2]{Kobayashi}
that $\eta$ extends smoothly from $\pi_M^{-1}(U_X)$ to a well-defined
connection $1$-form on all of~$M$ (still denoted by $\eta$). Since
\begin{equation}\label{relation1}
d\eta = 2\pi_M^*(\omega),
\end{equation}
it also follows that $\eta$ is a well-defined contact form on~$M$
(cf.~Section~\ref{SE.and.3S}).

We shall slightly abuse the notation by writing $\alpha_{i}$ also for the
lifting to $M$ via $\pi_M^*$ of the local $1$-forms \eqref{1-forms}.
Then $\alpha_1+i\alpha_2,\; \alpha_3+i\alpha_4,\;
\alpha_5+i\alpha_6$ can be considered as point-wise orthonormal $(1,0)$-forms 
on the complex cone $\pi^{-1}(U_X)\times\RE^+$ as $\pi_M$ 
is a Riemannian submersion and the associated bundle projection
$M\times\RE^+ \to X$ is holomorphic. 

The Riemannian cone $M\times\RE^+$ is simply-connected and Ricci-flat
K\"ahler. Therefore, $M\times\RE^+$ has a non-vanishing holomorphic
$(4,0)$-form $\widehat\Psi$, which is parallel with respect to the K\"ahler
metric. In local coordinates, the K\"ahler form on $M\times\RE^+$ is given by
$rdr\we\eta + r^2 \sum_{j=1}^3 \alpha_{2j-1}\we\alpha_{2j}$
(cf.~\eqref{K.cone}) and $\widehat\Psi$ can be written as
\begin{equation}\label{holo.vol}
\widehat\Psi = r^3 e^{i\mu} \left(dr+i\frac{\pi}{2}r\eta\right) \we (\alpha_1+i\alpha_2)\we
(\alpha_3 +i\alpha_4)\we (\alpha_5+i\alpha_6),
\end{equation}
where a smooth real function $\mu$ on $M$ is determined, up to
adding a constant, by the condition $d\widehat\Psi=0$. Setting
\begin{equation} \label{horizontal.3}
\Psi =
e^{-is}(\alpha_1+i\alpha_2)\we (\alpha_3+i\alpha_4)\we (\alpha_5+i\alpha_6),
\end{equation}
it is straightforward to check that
\begin{equation}\label{relation}
d\Psi = -2\pi i\,\Psi\we\eta,
\end{equation}
thus, noting also that $\Psi\we d\eta=0$, one can take $\mu=-s$
in~\eqref{holo.vol}. The holomorphic \mbox{$4$-form}
$\widehat\Psi$ extends from $\pi^{-1}(U_X)\times\RE^+$ to all of
$M\times\RE^+$ (e.g.\ by patching with a different choice of spherical 
coordinate neighbourhood in~\eqref{1-forms}). Respectively, $\Psi$ extends to
a well-defined horizontal complex volume form on~$M$.

We can now write a coordinate expression for the $S^1$-family of
nearly parallel $\G_2$-structures $\varphi_t$ on $M$ induced by the
Sasaki--Einstein structure, cf.~\eqref{def:G2form-sasaki-eintein},
\begin{align} \label{def:G2-M}
\varphi_t =  \frac\pi{2} \omega \wedge \eta + \cos (s+t)\, \Psi_+ + \sin (s+t)\, \Psi_-
=  \frac\pi{2} \omega \wedge \eta + \re (e^{-it} \Psi )
\end{align}
where the real $3$-forms $\Psi_+,\,\Psi_-$ in~\eqref{def:G2-M}
are determined by $\Psi = e^{-is}(\Psi_{+} + i \Psi_{-})$ and can be written
in coordinates using~\eqref{horizontal.3}. The respective $\G_2$ $4$-forms are
$$
\star_{\varphi_t}\varphi_t = \frac{1}{2} \omega\we\omega + (\cos (s+t) \,
\Psi_{-} -  \sin (s+t)\, \Psi_{+}) \wedge \eta  =  \frac{1}{2}
\omega\we\omega + \im (e^{-it} \Psi ) \wedge \eta,
$$
and from \eqref{relation1} and \eqref{relation} it follows that
\begin{equation}\label{4-form}
d\varphi_t = 2\pi  \star\varphi_t\,, 
\end{equation}
for each $t$.
The induced metric $g=g_{\varphi_t}$ on $M$ does not depend on $t$ and is given by
$$
g= \frac{1}{16} \Big(ds - \sum_{j=1}^{3} \cos\theta_j\, d \phi_j\Big)^2 + \frac{1}{4\pi}  \sum_{j=1}^{3} \Big(d\theta_{i}^2 + \sin^2\theta_j\, d \phi_j^2\Big),
$$
so the local coframe of $1$-forms  $\{\alpha_1, \ldots, \alpha_6,  \alpha_7=\eta\}$ is orthonormal.

\section{Associative 3-folds in Sasaki--Einstein $7$-manifolds} \label{3folds.SE} 

We now turn to consider minimal associative 3-folds in nearly parallel
$\G_2$-manifolds. In this section, we deal with the case when the nearly
parallel $\G_2$-structure arises from a Sasaki--Einstein structure on a
$7$-manifold, with particular attention to the regular Sasaki--Einstein
manifolds $S_k$ and $Q(1,1,1)$ 
discussed in Section~\ref{sect:formalG2} and Section~\ref{ex:nearlyparallel},
respectively.

Firstly, let us recall that if  $(P, \varphi)$ is a nearly-parallel $\G_2$-manifold, then an {\em associative $3$-fold}
is an oriented $3$-dimensional submanifold $X \subset P$ such that
$$\varphi|_X = \vol_X.$$

In particular, associative $3$-folds in $P$ are the links 
of Cayley cones in the metric cone over $P$, and hence are also minimal submanifolds of $P$. Even more,
Ball and Madnick in \cite[Theorem~2.18]{BallM} prove that the largest torsion
class of $\G_2$-structures for which every associative 3-fold is minimal is
given by the nearly parallel $G_2$-structures.

Let $(S,g_S)$ be a Sasaki--Einstein $7$-manifold and $\varphi_t$ the
corresponding $S^1$-family of nearly parallel $\G_2$-structures on $S$
defined in~\eqref{def:G2form-sasaki-eintein}. Assume that $S$ is
simply-connected, so the metric cone $S\times\RE^+$ has
holonomy in~$\SU(4)$, and denote by $\omega^c$ the K\"ahler form and by
$\widehat\Psi$ a non-vanishing holomorphic $4$-form on $S\times\RE^+$. Then
$S\times\RE^+$ has a 1-parameter family of torsion-free $\Spin(7)$-structures
induced by the closed $4$-forms
\begin{equation}\label{SU4-Spin7}
\widehat\Phi_t = \frac12 \omega^c\we\omega^c +
\re (e^{-it}\widehat\Psi), \quad  t\in\RE.
\end{equation}
cf.~\cite[Prop.~13.1.4]{Joyce}. Recall that we identify the $7$-manifold $S$
as a submanifold $S\times\{1\}$ of the cone. The nearly parallel
$\G_2$-structures $\varphi_t$ on~$S$ are then related to the
$\Spin(7)$-structures \eqref{SU4-Spin7} by
$$
\varphi_t=\partial_r\lrcorner\widehat\Phi_t|_{r=1},\quad
\text{for each } t,
$$
cf.~\cite[pp.~723--724]{AS}. If a $4$-dimensional submanifold $Z$ of
$S\times\RE^+$ is calibrated by $\widehat\Phi$ (i.e.\ $Z$ is a Cayley submanifold)
and $Z=Y\times\RE^+$ for some submanifold $Y\subset S$, then 
$(\varphi_t|_Y)\wedge dr$ is the volume form of the
conical metric $r^2 g_S+dr^2$ on $Y\times\RE^+$ as $\partial_r$ defines a unit
normal vector field along each $Y\times\{r\}$. It follows that $Y$ is an
associative 3-fold in~$S$. Conversely, if $Y\subset S$ is an associative
3-fold then $Z=Y\times\RE^+\subset S\times\RE^+$ is Cayley.

Examples of Cayley submanifolds, in the case when the $\Spin(7)$-structure
is of the form~\eqref{SU4-Spin7} induced by an $\SU(4)$-structure,
include complex surfaces and special Lagrangian submanifolds. We shall
consider these two possibilities in order.

If $Z=Y\times\RE^+$ is a complex surface in $S\times\RE^+$, so the tangent
spaces of $Z$ are preserved by the (integrable) almost complex structure $J$
on $S\times\RE^+$, then $Y$ is an `invariant' submanifold for the contact
structure on $S$ as defined in~\cite[\S 8.1]{Blair}. More explicitly, the
endomorphism $\Phi$ discussed in Section~\ref{SE.and.3S} maps each tangent
space of $Y$ into itself and the Reeb vector field~$\xi$ is tangent to~$Y$.

\begin{prop}\label{invar.assoc}
Let $S$ be a regular Sasaki--Einstein $7$-manifold with contact form
$\eta$ arising from a principal $S^1$-bundle $\pi:S\to X$ with Euler class
$c_1=[\omega]$, where $X$ is a projective complex 3-fold with K\"ahler
form $\omega$ and $d\eta=\pi^*(\omega)$. Let $\varphi_t$ be the corresponding
1-parameter family of nearly parallel $\G_2$-forms defined
in~\eqref{def:G2form-sasaki-eintein}.

Then, for each complex curve $\Sigma$ in $X$, the preimage
$Y_\Sigma=\pi^{-1}(\Sigma)\subset S$ is an invariant submanifold for the
contact structure $\eta$ and a (minimal) associative 3-fold with respect to
$\varphi_t$ for each~$t$.

In particular, if $S$ is $Q(1,1,1)$ or $S_k$, so $X = \CP^1\times P$,
with $P = \CP^1\times\CP^1$ or $P=P_k$ $(3\leq k\leq 8)$ a del Pezzo surface,
and $\omega = \omega_1 + \omega_P$, where $[\omega_1]$ is the (positive)
generator of $H^2(\CP^1,\bZ)$, $[\omega_P]\in H^2(P,\bZ)$
and $\Sigma = \CP^1\times\{p\}$, $p\in P$,
then the minimal associative $Y_\Sigma$ is diffeomorphic to the sphere~$S^3$.
\end{prop}

\begin{proof}
Recall from~\eqref{K.cone} that the K\"ahler form on the cone $S\times\RE^+$ is
\begin{equation}\label{CY.cone}
\omega^c=r^2d\eta + 2r\,dr\we\eta=r^2\pi^* \omega + 2r\,dr\we\eta,
\end{equation}
where $r\in\RE^+$ and we extended $\pi$ to a projection $\pi:S\times\RE^+\to
X$ independent of the $\RE^+$ factor. Pulling back to submanifolds
$Y_\Sigma\times\RE^+$ and $\Sigma\subset Y_\Sigma\times\{1\}$ we find that
$$
\frac12 (\omega^c\we\omega^c)|_{Y_\Sigma\times\RE^+} =
\pi^*(\omega|_{\Sigma\times\{1\}})\we 2\eta\we dr = \vol_{Y_\Sigma\times\RE^+},
$$
noting also that $\pi$ is a Riemannian submersion and $2\eta$ has unit length
in the metric induced by~$\varphi_t$. Thus, $Y_\Sigma\times\RE^+$ is calibrated
by $\frac12 \omega^c\we\omega^c$ and must be a complex surface in $S\times\RE^+$
by the Wirtinger inequality. The link $Y_\Sigma$ is therefore associative.

For the last part, we can consider $S\times\RE^+$ as the total space of a
holomorphic line bundle over~$X$. The restriction of this bundle to a
projective line $\Sigma$ is isomorphic to the hyperplane bundle 
$\mathcal{O}(1)$ over~$\CP^1$. The total space $Y_\Sigma\times\RE$ is then
biholomorphic to $\CP^2\setminus\{(0{:}0{:}1)\}$. We obtain that the
associative 3-fold $Y_\Sigma$ is diffeomorphic to the sphere $S^3$, noting
that the fibres of $Y_\Sigma\times\RE$ correspond to projective lines through
$(0{:}0{:}1)$, so the zero section $\Sigma$ can be identified as a projective
line in $\CP^2$ avoiding $(0{:}0{:}1)$. 
\end{proof}

\Needspace{3\baselineskip}
\begin{remarks}
\mbox{}
{\renewcommand{\labelenumi}{(\roman{enumi})}
\begin{enumerate}
\item
It is not difficult to see that every minimal associative
$Y_\Sigma$ in Proposition~\ref{invar.assoc} is invariant under the (isometric)
$S^1$-action on the principal bundle~$S$ and every deformation of the
holomorphic curve $\Sigma$ in $X$ induces an associative deformation
of~$Y_\Sigma$.

\item
We obtain from \eqref{def:G2form-sasaki-eintein} that
$\varphi_t|_{Y_\Sigma}=2\eta\we\Omega_\Sigma$, where
$\Omega_\Sigma=\pi^*(\omega|_\Sigma)$. As $\pi$ is a Riemannian submersion, the
form $\Omega_\Sigma$ at each point in $Y_\Sigma$ can be written as a wedge
product of two unit-length covectors which are orthogonal to $\eta$ in the
metric induced by~$\varphi_t$. (If $S=Q(1,1,1)$ and
$\Sigma=\CP^1\times\{(p_1,p_2)\}$, then
$\Omega_\Sigma=\alpha_1\wedge\alpha_2$ in the notation of~\eqref{1-forms}.)

\item
In the last part of Proposition~\ref{invar.assoc}, one can more generally take
$\Sigma$ to be the graph of a holomorphic embedding $\CP^1\to P$. For
$S=Q(1,1,1)$, the ambiguity of taking such~$\Sigma$ corresponds to a generic
choice of two rational functions of one complex variable. 
\end{enumerate}}
\end{remarks}

Concerning deformations of minimal associative 3-folds for nearly parallel
$\G_2$-structures of this type, we more generally have:

\begin{prop}\label{non-rigid}
Let $S$ be a regular Sasaki--Einstein $7$-manifold with contact form
$\eta$ arising from a principal $S^1$-bundle $\pi:S\to X$ and let
$\varphi_t$ be the induced $S^1$-family \eqref{def:G2form-sasaki-eintein} of
nearly parallel $\G_2$-structures on~$S$.

If $Y\subset S$ is an associative 3-fold with respect to $\varphi_{t_0}$ for
some fixed $t_0$ and $Y\neq \pi^{-1}(\Sigma)$ for any real $2$-dimensional
submanifold $\Sigma\subset X$,
then the free $S^1$-action on the principal bundle $S$ induces non-trivial
$\varphi_{t_0}$-associative deformations of $Y$ (in particular, $Y$ is
{\em not} rigid).
\end{prop}

\begin{proof}
The free $S^1$-action on~$S$ preserves the family $\varphi_t$ and all
$\varphi_t$ induce the same (Sasaki--Einstein) metric on~$S$. Therefore, the
$S^1$-orbit of $Y$ consists of volume minimizing submanifolds of~$S$ which are
in general distinct from~$Y$ by the hypothesis.

The result now follows by application of~\cite[Theorem~II.4.2]{HarveyLawson}.
\end{proof}

We now turn to consider associative 3-folds arising as links of Cayley
submanifolds $Z=Y\times\RE^+$ which are a special Lagrangian in $S\times\RE^+$,
so $\omega^c|_Z=0$ and $\re (e^{-it}\widehat\Psi)|_Z=0$ for some
fixed~$t$. Then (and only then) the cross-section $Y$ is, by definition,
a {\em special Legendrian} submanifold of the Sasaki--Einstein manifold~$S$.
Equivalently, $\eta|_Y=0$ and $\re \Psi_t|_Y=0$,
where $\Psi_t=\partial_r\lrcorner(e^{-it}\widehat\Psi)|_{r=1}$ is the
horizontal volume form, cf.~\cite[Prop.~3.3]{M2015}. (As before, $\eta$ is
the contact form on~$S$.) Since $Y$ is associative with respect to the
$\G_2$-structure $\varphi_t$, $Y$ is a minimal submanifold for the
Sasaki--Einstein metric on~$S$. Conversely, it is known that an oriented
Legendrian submanifold $Y\subset S$ is minimal if and only if $Y$ is special
Legendrian \cite[Prop.~3.2]{M2015}.

The following result will be useful for producing examples of associative
3-folds.

\begin{proposition}[{\cite[Prop.~3.4]{M2015}}]\label{involution}
Let $S$ be a Sasaki--Einstein $(2n+1)$-manifold and let
$\tau_S:S\times\RE^*\to S\times\RE^*$ be an anti-holomorphic involution
for the complex structure of the K\"ahler--Einstein cone $S\times\RE^*$
fixing the coordinate $r\in\RE^+$. If the fixed point set
$C_\tau$ of $\tau$ is not empty, then the link $C_\tau\cap (S\times\{1\})$
is a special Legendrian submanifold of $S$.
\end{proposition}

In the case of a regular Sasaki--Einstein manifold $S$ we have the following.

\begin{theorem}\label{Legendrian}
Let $S$ be a regular Sasaki--Einstein $7$-manifold with contact form
$\eta$ arising from a principal $S^1$-bundle $\pi:S\to X$ with Euler class
$c_1=[\omega]$, where $X$ is a K\"ahler--Einstein Fano 3-fold
with K\"ahler form $\omega$ and $d\eta=\pi^*(\omega)$. Let $\varphi_t$ be the
corresponding $1$-parameter family of nearly parallel $\G_2$-forms defined
in~\eqref{def:G2form-sasaki-eintein}.

Then for each compact special Legendrian submanifold $Y\subset S$, the
restriction \mbox{$\pi|_Y:Y\to Y_X$} is a finite covering of a Lagrangian
submanifold \mbox{$Y_X\subset X$}.

If $Y_X\subset X$ is a compact simply-connected Lagrangian submanifold,
thus a Lagrangian 3-sphere, then $Y_X$ lifts to an $S^1$-family of
Legendrian submanifolds $Y_s \subset S$ such that $\pi(Y_s)=Y_X$ for each
$s\in S^1$.

Assume that $\tau:X\to X$ is an isometric anti-holomorphic involution.
If the fixed point set $Y_X\subset X$ of $\tau$ is non-empty,
then $Y_X$ is Lagrangian and diffeomorphically lifts to a special Legendrian
(hence associative) submanifold of~$S$.
\end{theorem}

\begin{proof}
If a submanifold $Y$ is Legendrian, i.e.\ $\eta|_Y=0$, then $Y$ cannot be
tangent to any fibre in the principal circle bundle~$S$ and $\pi$ maps $Y$
locally diffeomorphically onto the image $\pi(Y)=Y_X$ which is a submanifold
of~$X$. Since $Y$, hence also $Y_X$, are compact we obtain that $\pi|_Y$ is a
finite cover. Considering \eqref{CY.cone} we find that $Y_X$ is Lagrangian
submanifold of $X$, $\omega|_{Y_X}=0$.

For the second claim, since $c_1=[\omega]$ vanishes on $Y_X$ and
$d\eta=\pi^*\omega$, the connection $\eta$ on the $S^1$-bundle $S$ restricts to
a flat connection over~$Y_X$. Furthermore, $\eta|_{Y_X}$ must be a trivial
product connection as $Y_X$ is simply connected. So
$\pi^{-1}(Y_X)\cong Y_X\times S^1$ with $\eta|_{Y_X}$ corresponding to $ds$,
$s\in S^1$, and $Y_s=Y_X\times\{s\}$ for each $s$, is a Legendrian submanifold
of~$S$.

The principal bundle $S$ is associated to the anticanonical bundle $K^{-1}_X$
and $S\times\RE^*$ is biholomorphic to the complement of the zero section.
The hypotheses on $\tau$ imply that $\tau^*\omega_X=-\omega_X$ and
$\tau^*\overline{K^{-1}_X}$ is a holomorphic line bundle isomorphic to
$K^{-1}_X$. We find that $\tau^*$ defines a lift of $\tau$ to an
antiholomorphic involution on~$K^{-1}_X$ and (by restriction) on
$S\times\RE^+$ with $\tau^*\eta=-\eta$ and $\tau\circ r=r$. Thus
$\tau^*$ preserves the K\"ahler-Einstein metric on $S\times\RE^+$ and the
fixed point set of $\tau^*$ is an oriented (hence trivial) real line bundle 
over~$Y_X$. The desired special Legendrian is obtained as the
intersection with $S=S\times\{1\}$ by application of
Proposition~\ref{involution}.
\end{proof}

If the Fano $3$-fold $X$ also appears as a smooth fibre in a Lefschetz fibration
$\lambda:E\to\Delta$ over a disc with $0\in\Delta$ the only critical value,
then the vanishing cycles in $X$ (the cycles that degenerate to a point as $X$
is deformed to a fibre over $0$) are represented by Lagrangian
$3$-spheres (see~\cite{arnold} or \cite[\S~4]{seidel}).

In the case when the nearly parallel $\G_2$-manifold $S$ is
$Q(1,1,1)$ or $S_3$ (see Theorem~\ref{th:pezzo-nonformal-sasa-einst}) we can
say more.

\begin{prop}\label{SL.torus}
Let $\pi_M:M=Q(1,1,1)\to X=S^2\times S^2\times S^2$ be
the principal
$S^1$-bundle \eqref{Q.bundle} and $\varphi_t$ the $1$-dimensional family of
nearly parallel $\G_2$-structures \eqref{def:G2-M} on~$M$. Let $L\subset X$ be
a 3-torus defined by $\theta_j=\pi/2$, $j=1,2,3$, in the spherical coordinates
$\phi_j,\theta_j$ on $X$ (see Section~\ref{ex:nearlyparallel}).

Then $L$ lifts via $\pi_M$ to a family of minimal Legendrian 3-tori
$L_s\subset M$, $s\in\RE/ 2\pi\bZ $. For each $s$, the 3-torus $L_s$ is
associative with respect to $\varphi_{t}$ for all~$t$.
\end{prop}

As $X$ and hence $Q(1,1,1)$ are toric manifolds, the
existence of a compact special Legendrian submanifold in $Q(1,1,1)$ follows
from the main result in~\cite{M2015}. However, the explicit examples of
special Legendrians in $Q(1,1,1)$ are not considered in~\cite{M2015}.
\begin{proof}[Proof of Proposition~\ref{SL.torus}]
It follows at once from the expression~\eqref{omega} for the K\"ahler form
that the 3-torus $L$ is Lagrangian in~$X$. We know from~\eqref{relation1} that
$\omega$ is proportional to the curvature of the principal $S^1$-bundle $M$,
thus the restriction of this bundle to $L$ has first Chern class zero and is a
trivial bundle. Furthermore, with $\theta_j=\pi/2$ the $1$-form $\eta$
in~\eqref{eta} defines a trivial product connection on $\pi_M^{-1}(L)$.
Thus $L$ lifts via $\pi$ to a family of horizontal 3-tori $L_s\subset M$,
parametrised by $s\in S^1$.

When $s+t=\pi/2$, from~\eqref{def:G2-M} and \eqref{horizontal.3} we obtain 
that $\varphi_t|_{L_s} = \Psi_-|_{L_s} = -\alpha_2\we\alpha_4\we\alpha_6$,
whence $L_s$ (with appropriate orientation) is $\varphi_t$-associative and
minimal Legendrian. Then $L_s$ is $\varphi_t$-associative, for all $t$,
by application of Proposition~\ref{non-rigid}.
\end{proof}

The Fano $3$-fold $X_3=P_3\times\CP^1$ is toric because $P_3$ is so. Recall that
the del Pezzo surface is the blow-up $P_3=\CP^2\#\, 3\overline{\CP^2}$ and is,
up to isomorphism, independent of the choice of three non-colinear
points in $\CP^2$. Also, $P_3$ admits a K\"ahler--Einstein metric (\cite{siu}
or \cite{TY}) which can be taken to be invariant under the torus action.
(The induced nearly parallel $\G_2$-structure $\varphi_t$
of \eqref{def:G2form-sasaki-eintein} on $S_3$ therefore is also torus-invariant
for all~$t$.)

We construct a special Legendrian associative $3$-fold in $S_3$
as a fixed point set of an anti-holomorphic involution. The existence of 
this special Legendrian is of course again a special case of the main result
in~\cite{M2015}, but is rather simpler than in a general case.

The complex surface $P_3$ can be identified as the simultaneous solution of
two complex bilinear equations
\begin{equation}\label{cp2.model}
\{(z,w)\in\CP^2\times\CP^2 : z_0 w_0 = z_1 w_1 = z_2 w_2 \}
\end{equation}
and the blow-up $P_3\to\CP^2$ is the restriction of the first projection.
The effective torus action on $P_3$ is given by
$(z_0,z_1,z_2,w_0,w_1,w_2)\mapsto
(\xi_0 z_0,\xi_1 z_1,\xi_2 z_2,\xi_0^{-1} w_0,\xi_1^{-1} w_1,\xi_2^{-1} w_2)$,
where $\xi_i\in\CX^*$, $\xi_0\xi_1\xi_2=1$, and this can be interpreted as an
embedding of $(\CX^*)^2$ in $P_3$ as an open dense subset. We deduce that
in the notation of~\eqref{cp2.model} the map
$\tau_3(z_i,w_i)=(\bar{w}_i,\bar{z}_i)$ induces an antiholomorphic isometry of
the K\"ahler--Einstein~$P_3$.
The fixed point set of this involution is diffeomorphic to the
$2$-torus~\cite{C.T.C.Wall}.

Extending to an antiholomorphic involution of the Fano 3-fold~$X_3$ with the
complex conjugation on the $\CP^1$-factor, we obtain, by application of
Theorem~\ref{Legendrian}.

\begin{proposition}\label{assoc.torus}
There exists an associative $3$-torus in the nearly parallel
$\G_2$-manifold $(S_3,\varphi_t)$.
\end{proposition}

\section{Associative $3$-folds in the Aloff--Wallach spaces} \label{assAW}

We now revisit the Aloff--Wallach spaces $W_{k,l}$, discussed in  Subsection~\ref{AW},
and find examples of associative submanifolds in $W_{k,l}$. Recall that $k, l$
are non-zero, co-prime integers. We begin by briefly recalling
from~\cite{BallOl, CMS} the construction of proper nearly parallel
$\G_2$-structures on $W_{k,l}$.

Let $\{e_i\}_{i=0,\ldots,7}$ be a basis of left-invariant vector fields on
$\SU(3)$ such that $e_0$ is everywhere tangent to the orbits of the
$S^1_{k,l}$ action. We shall interchangeably think of $e_i$ as elements of
the Lie algebra $\su(3)$ with $e_0\in\frs^1_{k,l}$.
Choose $\{e_i\}$ so that writing $\{e^i\}$ for the dual basis of $\su(3)^*$
(or the left-invariant $1$-forms) the Maurer--Cartan form
$\Omega = \sum_i e_i e^i$ on $\SU(3)$ with values in $\mathfrak{su}(3)$ is
expressed as
\begin{equation}\label{MC}
\Omega = \frac1{\sqrt 2}
\begin{pmatrix}
i (\frac{k}{\sqrt{3}\, s} e^0 + \frac{l-m}{3s} e^4) & e^1 + ie^5 & -e^3 + ie^7\\
-e^1 + ie^5 & i (\frac{l}{\sqrt{3}\, s} e^0 + \frac{m-k}{3s}e^4) & e^2 + ie^6\\
e^3 + ie^7 & -e^2 + ie^6 & i (\frac{m}{\sqrt{3}\, s} e^0 + \frac{k-l}{3s}e^4)
\end{pmatrix},
\end{equation}
where $m=-k-l$ and $s=\sqrt{(k^2+l^2+m^2)/6}$.

Then for any choice of non-zero constants $A,B,C,D$ the $3$-form
\begin{equation}\label{phi_W}
\varphi_W = ABC (e^{123} - e^{167} +  e^{257} - e^{356})
- D (A^2 e^{15} + B^2 e^{26} + C^2 e^{37}) \wedge e^4
\end{equation}
descends to a well-defined positive $3$-form giving a coclosed $\G_2$-structure on
$W_{k,l}$. The induced orientation corresponds to a non-vanishing
7-form $De^{1234567}$ and the coframe $Ae^1,Be^2,Ce^3,De^4,Ae^5,Be^6,Ce^7$
of $1$-forms is orthonormal in the induced metric $g_{\varphi_W}$.
Note that we may assume without loss of generality that $A>0$ and $B>0$.
From the results in~\cite{CMS} it follows that
if $\{k,l,m\}$ are not $\{1,1,-2\}$ or $\{1,-1,0\}$,
then every homogeneous, coclosed 
$\G_2$-structure on $W_{k,l}$  is of the form~\eqref{phi_W}. Furthermore, by
using the Maurer--Cartan equation $d\Omega=-\frac12 [\Omega,\Omega]$, it was
proved that exactly two such $\G_2$-structures are nearly parallel, more
precisely, these are proper nearly parallel.

Recall from Subsection \ref{AW} that each generic Aloff--Wallach space
$W_{k,l}$ has a structure of a smooth fibre bundle $\pi_{k,l}:W_{k,l}\to\CP^2$
with typical fibre the spherical space form $S^3/\Z_{|k+l|}$. This fibration
is not unique. The Weyl group of $\SU(3)$ contains an element of order 3 which
induces a diffeomorphism $\upsilon: W_{k,l}\to W_{l,m}$. The composition
$\pi_{l,m}\circ\upsilon$ defines a different fibration, in general by
different spherical space forms.

\begin{theorem}\label{assoc.fibr}
Let $\varphi_W$ be a nearly parallel $\G_2$-structure of the form~\eqref{phi_W}
on an Aloff--Wallach space $W_{k,l}$, with $k,l$
non-zero and co-prime integers such that $\{k,l,m\}$ are not $\{1,1,-2\}$ or
$\{1,-1,0\}$. Then the fibres of $\pi_{k,l}$ are embedded minimal associative
3-folds with respect to~$\varphi_W$.

Furthermore, for suitably `generic' $k,l$, the Aloff--Wallach space $W_{k,l}$
has three different $4$-dimensional deformation families of minimal associative
spherical space forms.
\end{theorem}

\begin{proof}
The fibre bundle $\pi_{k,l}$ can be obtained by considering the embedding of
$\U(2)$ as a Lie subgroup of $\SU(3)$ consisting of the block-diagonal
matrices with blocks $Ae^{i\theta}$ and $1/\det (Ae^{i\theta})=e^{-2i\theta}$,
for all $A\in \SU(2)$ and $\theta\in\RE$. Then
$$
\pi_{k,l}: W_{k,l}=SU(3)\S^1_{k,l} \to \SU(3)/\U(2) \cong \CP^2
$$
between respective cosets and the fibres is
$\U(2)/S^1_{k,l}\cong S^3/\Z_{|k+l|}$.

Comparing the tangent spaces to the cosets of $\U(2)$ in $\SU(3)$
with~\eqref{MC}, we find that the vertical spaces of the principal $\U(2)$-bundle 
$\SU(3)\to\CP^2$ are defined by the vanishing of $e^2,e^3,e^6,e^7$.
Thus $\varphi_W$ restricts on each fibre of $\pi_{k,l}$ to the volume form
$A^2D e^{145}$ of the metric induced by $g_{\varphi_W}$, thus the fibres are
associative.

We next show that the fibres of $\pi_{l,m}\circ\upsilon$ are also minimal
associative for the {\em same} $\G_2$-structure $\varphi_W$.

To this end, we further note that $\pi_{k,l}$ factors through the flag
manifold $F(1,2)\cong \SU(3)/T^2$ (where $T^2\subset \SU(3)$ is the subgroup of
diagonal matrices),
\begin{equation}\label{u1.bundle}
W_{k,l} \to F(1,2) \to \CP^2.
\end{equation}
Here the first map is a principal $S^1$-bundle and the second map is a
$\CP^1$-bundle associated to the principal $\U(2)$-bundle $\SU(3)\to\CP^2$
discussed above. In particular, the vertical space of the $S^1$-bundle is
spanned by $e_4$.

It is well-known that $F(1,2)$ is a complex manifold with a K\"ahler
structure defined by the Kirillov--Kostant--Souriau symplectic form.
We can consider a different choice of the $\CP^1$-bundle $h_3:F(1,2) \to \CP^2$
corresponding to a different embedding of $\U(2)$ as a Lie subgroup of $\SU(3)$,
with the tangent space at the identity now spanned by
$e_0,e_3,e_4,e_7\in\su(3)$ (rather than $e_0,e_1,e_4,e_5$ chosen above). Let
$Y$ be a fibre of the map $W_{k,l} \to F(1,2) \overset{h_3}{\to} \CP^2$. Then
the forms $e^1,e^5,e^2,e^6$ vanish on $Y$ and so $\varphi_W|_Y = C^2D e^{347}$.
Thus $Y$ is an embedded minimal associative 3-fold with respect to a
$\G_2$-structure $\varphi_W$ of the form~\eqref{phi_W} on $W_{k,l}$.

The third family of associative 3-folds follows by a similar argument,
replacing the $\CP^1$-bundle $h_3$ with $h_2:F(1,2) \to \CP^2$ where the
vertical spaces  are now spanned by $e_0,e_2,e_4,e_6$.
\end{proof}

\begin{remark}
The exceptional Aloff--Wallach space $W_{1,-1}$ has, up to homotheties,  only
one homogeneous nearly parallel $\G_2$-structure, which is of the form
~\eqref{phi_W}. The argument and result of Theorem~\ref{assoc.fibr}
applies to this latter $\G_2$-structure noting that the fibres of $\pi_{1,-1}$
are now $S^2\times S^1$ as $S^1_{1,-1}$ is a subgroup of $\SU(2)$.

When the $G_2$-structure $\varphi_W$ in Theorem~\ref{assoc.fibr} is
{\em not} nearly parallel, the fibres of $\pi_{k,l}$ are associative 3-folds
but need not be minimal.
\end{remark}

\begin{remark}
For the exceptional Aloff--Wallach space $W_{1,1}$ there are still exactly two
homogeneous Einstein metrics~\cite{Ni}. However, one of these metrics is
induced by a proper nearly parallel $G_2$-structure on $W_{1,1}$ which is
{\em not} of the form \eqref{phi_W}, whereas the other metric is induced by
a $3$-Sasakian structure on $W_{1,1}$. For this latter $3$-Sasakian manifold, the
above construction of minimal associative $3$-folds based on~\eqref{u1.bundle}
remains valid. On the other hand, very recently Ball and Madnick
constructed in $W_{1,1}$ examples of associative 3-folds diffeomorphic to
an $S^1$-bundle over a genus $g$ surface for all $g\ge 0$
\cite[Theorem 5.9]{BallM3}.
\end{remark}

\section{Locally conformal parallel $\Spin(7)$-structures} \label{sect:Spin(7)}

In this section we give examples of  formal compact 8-manifolds,  
 with a locally conformal parallel $\Spin(7)$-structure.

An $8$-dimensional manifold $N$ has a $\Spin(7)$-structure if there is a reduction of the structure group of 
its frame bundle from ${\GL}(8,\mathbb{R})$ to the  exceptional Lie group $\Spin(7)$.  
In opposite to the existence of $\G_2$-structures on any orientable and spin manifold of dimension $7$, it happens that not every
$8$-dimensional spin manifold $N$ admits a $\Spin(7)$-structure. In fact, in \cite{LM} it is proved that $N$ has a $\Spin(7)$-structure 
if and only if 
$$
p_{1}^{2}(N) - 4\ p_2(N) + 8 \chi(N) =0,
$$
for an appropriate choice of the orientation, where $p_1(N)$, $p_2(N)$ and $\chi(N)$ are the first Pontryagin class, the 
second Pontryagin class and the Euler characteristic of $N$, respectively.

The presence of a $\Spin(7)$-structure is equivalent to the existence of a nowhere vanishing global differential $4$-form $\Omega$  on the $8$-manifold $N$,
which can be written locally as
 \begin{align*}
\Omega=\Big(e^{127}+e^{347}+e^{567}+e^{135}-e^{146}-e^{236}-e^{245}\Big)\wedge e^8 \\
+\, e^{1234} + e^{1256} + e^{1367} + e^{1457} + e^{2357} + e^{2467} + e^{3456}, 
\end{align*}
with respect to some (local) basis $\{e^1,\dotsc, e^8\}$ of the (local) $1$-forms on $N$. 

Since $\Spin(7) \subset \SO(8)$, a $\Spin(7)$-structure $\Omega$ on $N$ determines a Riemannian metric $g_\Omega$ and an orientation on $N$ 
such that 
$$
g_{\Omega}(U,V) \vol=\frac 17 (U\lrcorner\Omega) \wedge \star_{8}(V\lrcorner\Omega),
$$
for any vector fields $U, V$ on $N$, where $\vol$ is the volume form on $N$, and $\star_{8}$ is the Hodge star operator determined by $g_{\Omega}$.
Note that if $\Omega$ is a $\Spin(7)$-structure on $N$, then $\Omega$ is a self-dual $4$-form, i.e.\ $\star_{8} \Omega = \Omega$.

Let us recall that a $\Spin(7)$-structure $\Omega$ on a 8-manifold $N$ is said to be {\em parallel} if the induced metric by  $\Omega$ has holonomy 
contained in $\Spin(7)$. This is equivalent to say that the $4$-form $\Omega$ is parallel with respect to the Levi--Civita connection of the metric
$g_{\Omega}$, which happens if and only if $d\Omega=0$ \cite{Fer}.

\begin{definition} \label{def:lcSpin(7)}
A $\Spin(7)$-structure $\Omega$ on a 8-manifold $N$ is said to be {\em locally conformal parallel} if 
  \begin{equation}\label{lcp-spin(7)}
 d\Omega=\Omega \wedge \Theta,
  \end{equation}
 for a closed non-vanishing $1$-form $\Theta$, which is known as the {\em Lee form} of the $\Spin(7)$-structure. A manifold endowed with such a structure is called
{\em locally conformal parallel $\Spin(7)$-manifold}.
\end{definition}

Let $\varphi$ be a nearly parallel $\G_2$-structure on a $7$-manifold $P$, so that $d\varphi=\tau_0 \star\varphi$ by \eqref{g2np}. Then, the product 
manifold $P\times S^1$ carries a natural locally conformal parallel $\Spin(7)$-structure $\Omega$ defined by 
  \begin{equation}\label{lcp-spin(7)-product}
\Omega = \varphi \wedge  \theta + \star \varphi,
  \end{equation}
where $\theta$ is the volume form on $S^1$, and $\star$ is the Hodge star operator determined by the induced metric by $\varphi$ on $P$. Note that
 the Lee form $\Theta$ is given by  $\Theta= \tau_0\,\theta$. In fact, using \eqref{lcp-spin(7)-product},
 we have $d\Omega=\tau_0\,\star\varphi \wedge  \theta = \tau_0\,\Omega \wedge  \theta$, that is $d\Omega = \Omega \wedge  \Theta$.

By \cite[Theorem 3.2]{BFMT},  we know that $Q(1,1,1)$ is non-formal. Hence the product manifold $Q(1,1,1)\times S^1$ is non-formal
by \cite{DGMS}. Then, if we consider the $S^1$-family of nearly parallel $\G_2$-structures 
$\varphi_t$ defined in \eqref{def:G2-M}, we have:

\begin{proposition}\label{prod.non-formal}
The product manifold $Q(1,1,1)\times S^1$ is non-formal and has an  $S^1$-family  of locally conformal parallel $\Spin(7)$-structures 
$\Omega_{t} = \varphi_{t} \wedge  \theta + \star \varphi_{t}$, with Lee form $\Theta= 4\,\theta$.
\end{proposition}

Let $P$ be a differentiable manifold and let $\rho:P\to P$ be a diffeomorphism.
The {\em mapping torus} $P_{\rho}$ of $\rho$ is the
manifold obtained from $P\times[0,1]$ by identifying the ends with $\rho$, that is
$$
P_{\rho}= \frac {P\times[0,1]}{(x,0) \sim (\rho(x),1)}.
$$
The natural
map $\pi \colon P_{\rho} \to S^1$ defined by $\pi(x,t)=e^{2\pi it}$ is the projection
of a locally trivial fiber bundle (here we think of $S^1$ as the interval $[0,1]$ with identified end points). Thus, any  $\rho$-invariant 
form $\beta$ on $P$ defines a form on $P_{\rho}$, 
since the pullback of $\beta$ to $P\times\mathbb{R}$ is invariant under the diffeomorphism $(x,t) \mapsto  (\rho(x),t+1)$. 
For the same reason, the $1$-form $dt$ on $\RE$, where $t$ is the coordinate on $\RE$,  induces a closed $1$-form $\nu$ on $P_{\rho}$.

A theorem of Tischler \cite{Tischler} asserts that a compact manifold is a mapping torus if and only if it admits a non-vanishing closed $1$-form.
This result was extended to locally conformal parallel $\Spin(7)$ manifolds 
in {\cite[Theorem B]{IPP}. There, it is proved that
there exists a} fibre bundle $N\rightarrow S^1$ with abstract fibre $P/\Gamma$, where $P$ is a compact simply connected 
nearly parallel $\G_2$-manifold and $\Gamma$ is a finite subgroup of isometries of $P$ acting freely. Moreover, the cone of
$P/\Gamma$ covers $N$ with cyclic infinite covering transformation group.

Notice that if $P$ is a $7$-dimensional compact manifold endowed with a nearly parallel $\G_2$-structure $\varphi$,  
and $\rho : P \rightarrow P$ is a diffeomorphism such that $\rho^*\varphi =
\varphi$,  
then $\rho$ preserves the orientation and the metric on $P$ induced by the $\G_2$-structure $\varphi$. So, 
$\rho^*(\star\varphi) = \star\varphi$, and
the mapping torus $N=P_\rho$ has a locally conformal parallel $\Spin(7)$-structure $\Omega$ given by 
$
\Omega = \varphi \wedge \nu + \star\varphi.
$

Let us recall that if $P$ is a differentiable manifold and $\rho:P\to P$ is a diffeomorphism,
then the cohomology of the mapping torus $N=P_{\rho}$ of $\rho$ sits in an exact sequence \cite[Lemma 12]{BFM}
 $$
 0 \to C^{r-1} \to H^r(N) \to K^{r}\to 0,
 $$
where $K^r$ is the kernel of $\varphi^*-\id\colon H^{r}(P)\to  H^{r}(P)$, and $C^{r}$ is its cokernel. Thus,
\begin{align*}
H^{r}(P_{\rho})\cong&\ker\left(\varphi^{*}-\id:H^r(P)\, \rightarrow\,   H^r(P)\right)\\
&\oplus[\nu]\wedge\frac{H^{r-1}(P)}{{\mathrm {Im}}\left(\varphi^{*}-\id:H^{r-1}(P)\, \rightarrow\,   H^{r-1}(P)\right)}\,.
\end{align*}

Now we consider the manifold $M=Q(1,1,1)$ with the family of nearly parallel $\G_2$-structures $\varphi_t$ defined by \eqref{def:G2-M}.
Let  $\rho$ be the diffeomorphism of $Q(1,1,1)$ given by
$$
\rho(\theta_1, \phi_1, \theta_2, \phi_2,   \theta_3, \phi_3, s) = ( \theta_2, \phi_2, - \theta_1, \phi_1,  \theta_3, \phi_3, s).
$$
Then, the diffeomorphism $\rho$ on the $1$-forms  $\{\alpha_1, \ldots, \alpha_6, \alpha_7=\zeta\}$ on $Q(1,1,1)$ is given by
$$
 \rho^*\alpha_1= \alpha_3, \quad   \rho^*\alpha_2= \alpha_4, \quad  \rho^*\alpha_3= -\alpha_1, \quad  \rho^*\alpha_4= - \alpha_2, \quad  \rho^*\alpha_i= \alpha_i,  \,\,  i= 5, 6, 7.
$$

\begin{proposition}\label{prop:ex-mappingtorus}
The mapping torus $M_{\rho} = Q(1,1,1)_{\rho}$ is formal and it has an $S^1$ family 
of locally conformal parallel $\Spin(7)$-structures.
\end{proposition}

\begin{proof} 
Clearly the diffeomorphism $\rho$ preserves the nearly parallel $\G_2$-structures
$\varphi_t$ defined in \eqref{def:G2-M} and, taking into account \eqref{4-form},  $\rho$ preserves also $\star\varphi_t$.
Then, by \cite[Theorem B]{IPP}, we know that $M_{\rho} = Q(1,1,1)_{\rho}$ carries a $S^1$-family   
of locally conformal parallel $\Spin(7)$-structures.

In order to prove that $M_{\rho} = Q(1,1,1)_{\rho}$ is formal we proceed as follows.
As in Section \ref{ex:nearlyparallel}, we write with the same symbol the lifting to $M$ of the $1$-forms 
 $\alpha_{i}$, $1\leq i \leq 6$, and of the K\"ahler forms $\omega_j$, $1\leq j \leq 3$, where $[\omega_1]$,  $[\omega_2]$ and $[\omega_3]$ 
 are the generators of the integral cohomology group of each of the 
$S^2$-factors on $S^2\times S^2 \times S^2$. 
Now, for simplicity of notation we write $a_1=[\omega_1]$,  $a_2=[\omega_2]$ and $a_3=[\omega_3]$. Since $M= Q(1,1,1)$ is the principal $S^1$-bundle 
 $$
 S^1 \too M= Q(1,1,1) \too S^2\times S^2\times S^2
 $$ 
with first Chern class equal to $a_1+a_2+a_3$, the Gysin sequence gives that
 \begin{align*}
& H^0(M,{\mathbb{Z}}) =H^7(M,\bZ)=\bZ, \\ 
& H^1(M,{\mathbb{Z}})=H^3(M,\bZ)=H^6(M,\bZ)=0, \\
& H^2(M,{\mathbb{Z}})=H^5(M,\bZ)=\bZ^2\, , \\ 
& H^4(M,{\mathbb{Z}})=\bZ \la a_1a_2,a_1a_3,a_2a_3\ra/\la a_1a_2+a_1a_3,a_2a_1+a_2a_3,a_3a_1+a_3a_2 \ra =\bZ_{2}.
 \end{align*}
 So, the de Rham cohomology groups of $M=Q(1,1,1)$ up to the degree $4$ are 
$$
 H^0(M)=\langle 1\rangle\,, \quad  H^1(M)= H^3(M)\,=\,H^4(M)\,=\,0\,, \quad  H^2(M)= \langle a_1,\, a_2\rangle\,.
$$
Since  $\rho^*\omega_1=\omega_2$ and $\rho^*\omega_2=\omega_1$, the de Rham cohomology groups of the mapping torus $M_{\rho} = Q(1,1,1)_{\rho}$ up to the degree $4$ are 
\begin{align*}
H^0(M_{\rho})&=\langle 1\rangle, \quad H^1(M_{\rho})=\langle [\nu] \rangle, \quad H^2(M_{\rho})\,=\langle a_1+a_2\rangle,\\
H^3(M_{\rho})&=\langle (a_1+a_2)\wedge [\nu]\rangle, \quad H^4(M_{\rho})\,=\,0\,.
\end{align*}
Therefore, the minimal model of $M_{\rho}$ must be a differential graded algebra
$(\bigwedge V,d)$, being $\bigwedge V$ the free algebra of the form
$\bigwedge V=\bigwedge(a, b, x)\otimes \bigwedge V^{\geq 4}$, where $|a|=1$, $|b|=2$,
$|x|=3$, and $d$ is defined by $da=0=db$, $dx=b^2$. 
According to Definition~\ref{def:primera}, we get
$N^j=0$ for $j=1,2$, thus $M_{\rho}$ is $2$-formal. Moreover, $M_{\rho}$ is 
$3$-formal. In  fact, take $z\in I(N^{\leq 3})$ a closed
element in $\bigwedge V$. As $H^*(\bigwedge V)=H^*(M_{\rho})$ has non-zero cohomology in degrees $0,1,2,3,5,6,7,8$, 
it must be $\deg z=5,6,7,8$. If $\deg z=5$ then $z=b\cdot\,x$ (up to non-zero scalar), which is not closed because $d(b\cdot\,x) = b^3 \not=0$. If 
$\deg z=6$ then $z=a\cdot\,b\cdot\,x$ which is not closed because $d(a\cdot\,b\cdot\,x) = -a\cdot\,b^3 \not=0$.  If 
$\deg z=7$ then $z=b^2\cdot\,x$ which is not closed, and if $\deg z=8$ then $z=a\cdot\,b^2\cdot\,x$ which is not closed either.
Thus,  according to Definition \ref{def:primera}, 
$M_{\rho}$ is $3$-formal and, by Theorem \ref{fm2:criterio2},  $M_{\rho}$ is formal.
\end{proof}

Next, for $M=\mathcal B$ or $M=W_{k,l}$, where $\mathcal B =\SO(5)/\SO(3)$ 
is the Berger space and $W_{k,l}=\SU(3)/ S_{k,l}^1$ are the the Aloff--Wallach spaces,
we study the formality of the mapping torus of a diffeomorphism of $M$ preserving a nearly-parallel $\G_2$-structure on $M$ (see subsections \ref{sec:extra-1} and \ref{AW}).

\begin{proposition}\label{prop:maptorus-Berger}
Let $\rho : \mathcal B \rightarrow \mathcal B$ be a diffeomorphism preserving a nearly-parallel $\G_2$-structure on $\mathcal B$.
Then, the mapping torus ${\mathcal B}_{\rho}$ is formal and it has a locally conformal parallel $\Spin(7)$-structure.
\end{proposition}

\begin{proof} 
We consider a nearly parallel $\G_2$-structure on $\mathcal B$, which is preserved by the diffeomorphism $\rho : \mathcal B \rightarrow \mathcal B$.
(As we already mentioned in subsection \ref{sec:extra-1}, the only known nearly parallel $\G_2$-structure on $\mathcal B$ is proper and is explicitly given in \cite{BallM2}.) 
Theorem B in \cite{IPP} implies that ${\mathcal B}_{\rho}$ has a locally conformal parallel $\Spin(7)$-structure.

According to the proof of Theorem \ref{th:formal-Berger}, we know that $\mathcal B$ and the $7$-sphere $S^7$ have the same minimal model.
Hence, the de Rham cohomology groups of $\mathcal B$ up to the degree $6$ are 
$$
 H^0(\mathcal B)=\langle 1\rangle\,, \quad  H^k(\mathcal B)\,=\,0\,, \quad  1\leq k\leq 6.
$$
So, the de Rham cohomology groups of the mapping torus ${\mathcal B}_{\rho}$ up to the degree $6$ are 
$$
 H^0({\mathcal B}_{\rho})=\langle 1\rangle\,, \quad  H^1({\mathcal B}_{\rho})\,=\,\langle [\nu] \rangle\,,\quad  H^k({\mathcal B}_{\rho})\,=\,0\,, \quad  2\leq k\leq 6.
$$
Therefore, the minimal model of ${\mathcal B}_{\rho}$ must be a differential graded algebra
$(\bigwedge V, d)$, being $\bigwedge V$ the free algebra of the form $\bigwedge V=\bigwedge (a)\otimes \bigwedge V^{\geq 7}$,
with $|a|=1$, and $d$ is given by $da=0$. 
Now, 
we have $C^1=\la a\ra$, $C^i= 0$ for $2\leq i \leq 4$, and $N^j=0$ for $1\leq j \leq 6$. Thus,
Definition \ref{def:primera} 
implies that ${\mathcal B}_{\rho}$ is $3$-formal and, by Theorem \ref{fm2:criterio2}, ${\mathcal B}_{\rho}$ is formal.
\end{proof}

\begin{proposition}\label{prop:maptorus-AloffW}
Let $\rho : W_{k,l} \rightarrow W_{k,l}$ be a diffeomorphism preserving a nearly parallel $\G_2$-structure (not necessarily proper) on $W_{k,l}$.
Then, the mapping torus ${(W_{k,l})}_{\rho}$ is formal and it has a locally conformal parallel $\Spin(7)$-structure.
\end{proposition}

\begin{proof} 
If $\rho : W_{k,l} \rightarrow W_{k,l}$ is a diffeomorphism preserving a nearly parallel $\G_2$-structure on $W_{k,l}$, the mapping torus 
${(W_{k,l})}_{\rho}$ has a locally conformal parallel $\Spin(7)$-structure by \cite[Theorem B]{IPP}.

In order to prove that ${(W_{k,l})}_{\rho}$ is formal we proceed as follows. By 
Theorem \ref{th:formal-AlofW1} and Theorem \ref{th:formal-AlofW2}, we know that
$W_{k,l}$ and the product manifold
$S^2 \times S^5$ have the same minimal model, and hence the same de Rham cohomology. So, the de Rham cohomology groups of $W_{k,l}$ are: 
\begin{align*}
&H^0(W_{k,l})=\langle 1\rangle\,, \quad  H^1(W_{k,l})=\,H^3(W_{k,l})\,=\,H^4(W_{k,l})\,=\,H^6(W_{k,l})\,=\,0\,, \nonumber\\
&H^2(W_{k,l})= \langle \xi \rangle\,, \quad   H^5(W_{k,l})= \langle \tau \rangle\,,\quad   H^7(W_{k,l})= \langle \xi\wedge \tau \rangle. \nonumber 
\end{align*}
Let us consider the map $\rho^* : H^2(W_{k,l}) \rightarrow H^2(W_{k,l})$. It is clear that either $\rho^*(\xi )\not=\xi$ or $\rho^*(\xi )=\xi$.

We deal first with the possibility that $\rho^*(\xi )\not=\xi$. Then, the cohomology of ${(W_{k,l})}_{\rho}$ up to the degree $4$ is
$$
H^{0}({(W_{k,l})}_{\rho})=\langle 1\rangle\,, \qquad H^{1}({(W_{k,l})}_{\rho})=\langle [\nu] \rangle\,, 
 \qquad H^i({(W_{k,l})}_{\rho}) =0\,, \,\, 2 \leq i \leq 4.
$$
Therefore, the minimal model of ${(W_{k,l})}_{\rho}$ must be a differential graded algebra
$(\bigwedge V,d)$, being $\bigwedge V$ the free algebra of the form
$\bigwedge V=\bigwedge(a)\otimes \bigwedge V^{\geq 5}$, where $|a|=1$,
and $d$ is defined by $da=0$.
According to Definition~\ref{def:primera}, we get  $C^1=\la a\ra$, $C^i= 0$ for $2\leq i \leq 4$, and  
$N^j=0$ for $1 \leq j\leq 4$. Hence, by Definition \ref{def:primera}, ${(W_{k,l})}_{\rho}$ is $3$-formal and, 
by Theorem \ref{fm2:criterio2},
${(W_{k,l})}_{\rho}$ is formal.

Suppose now that $\rho^*(\xi )=\xi$. In this case, the cohomology of ${(W_{k,l})}_{\rho}$ up to the degree $4$ is
\begin{align*}
&H^0({(W_{k,l})}_{\rho})=\langle 1\rangle\,, \qquad \qquad \, H^{1}({(W_{k,l})}_{\rho})=\langle [\nu] \rangle\,, \qquad  H^{2}({(W_{k,l})}_{\rho})=\langle \xi \rangle\,,  \nonumber\\
&H^3({(W_{k,l})}_{\rho})=\langle [\nu] \wedge \xi \rangle\,, \qquad  H^{4}({(W_{k,l})}_{\rho})=0. \nonumber 
\end{align*}
Thus, the minimal model of ${(W_{k,l})}_{\rho}$ must be a differential graded algebra
$(\bigwedge V,d)$, being $\bigwedge V$ the free algebra of the form
$\bigwedge V=\bigwedge(a, b, x)\otimes \bigwedge V^{\geq 4}$, where $|a|=1$, $|b|=2$, $|x|=3$,
and $d$ is defined by $da=db=0$ and $dx=b^{2}$.
According to Definition~\ref{def:primera}, we get 
$N^j=0$ for $j\leq 2$, thus ${(W_{k,l})}_{\rho}$ is $2$-formal.
Moreover, ${(W_{k,l})}_{\rho}$ is $3$-formal. In  fact, take $\alpha\in I(N^{\leq 3})$ a closed
element in $\bigwedge V$. As $H^*(\bigwedge V)=H^*({(W_{k,l})}_{\rho})$ has cohomology in all the degrees except 
in degree $4$, and since $|\alpha| \geq 4$,
it must be $|\alpha|=5, 6, 7, 8$. If $|\alpha|=5$, then $\alpha=b\cdot\,x$ (up to non-zero scalar), 
which is not closed. If $|\alpha|=6$, then $\alpha=a\cdot\,b\cdot\,x$, which is not closed either.
If $|\alpha|=7$, then $\alpha=b^2\cdot\,x$, and if $|\alpha|=8$, then $\alpha=a\cdot\,b^2\cdot\,x$, but $\alpha$ is not closed in either case.
So,  according to 
Definition \ref{def:primera}, 
${(W_{k,l})}_{\rho}$ is $3$-formal and,
by Theorem \ref{fm2:criterio2}, ${(W_{k,l})}_{\rho}$ is formal.
\end{proof}


\subsection*{Acknowledgements}  We are very grateful to Dieter Kotschick and  the anonymous referees for useful comments.
The first author was partially supported 
by the Basque Government Grant IT1094-16 and by the Grant PGC2018-098409-B-100 of 
the Ministerio de Ciencia, Innovaci\'on y Universidades of Spain. The second author is  supported by the project 
PRIN  \lq \lq Real and complex manifolds: Topology, Geometry and Holomorphic Dynamics", by GNSAGA of INdAM and by a grant from the Simons Foundation (\#944448).
The fourth author was partially supported by Project MINECO (Spain)  PID2020-118452GB-I00.

\end{document}